\documentclass{amsart}

\usepackage{amsmath,amsfonts,amssymb,amsthm,amscd,comment,euscript,mathtools}
\usepackage{stmaryrd}
\usepackage[all,cmtip]{xy}
\usepackage{graphicx}
\usepackage{enumitem} %to change indendation of item lists
\usepackage[colorlinks=true]{hyperref} %to generate pdf with links
\usepackage[usenames,dvipsnames]{xcolor}

\allowdisplaybreaks[4]

\newcommand{\ie}{i.e.\ }

\newcommand{\Kr}{\delta^{\text{Kr}}}               % Kronecker delta 
\newcommand{\tooby}[1]{\stackrel{#1}{\longrightarrow}} %left arrow with upper label
\newcommand{\isoto}{\buildrel \sim\over\to}
\newcommand{\isofrom}{\buildrel \sim\over\leftarrow}

\newcommand{\hook}{\hookrightarrow}

\newcommand{\ZZ}{\mathbb{Z}}               % integers
\newcommand{\id}{\mathrm{id}}           % identity morphism

\DeclareMathOperator{\im}{\mathrm{im}}         % image
\DeclareMathOperator{\Hom}{\mathrm{Hom}}

\DeclareMathOperator{\Spec}{\mathrm{Spec}}         % Spec 
\newcommand{\cO}{\mathcal{O}}      % structural sheaf
\newcommand{\cL}{\mathcal{L}}       % some sheaf of rank one
\newcommand{\Gm}{\mathbb{G}_m}      % group scheme Gm
\newcommand{\PP}{\mathbb{P}}      % projective space 
\renewcommand{\AA}{\mathbb{A}}      % affine space 

\newcommand{\TB}[1]{\mathcal{T}_{#1}}     %tangent bundle 
\newcommand{\NB}[2]{\mathcal{N}_{#1}#2}     %normal bundle

\newcommand{\RS}{\Sigma}    % set of roots
\newcommand{\cl}{\Lambda}  % character lattice 
\newcommand{\tor}{\mathfrak{t}} % torsion index of the root datum 

\newcommand{\LL}{\mathbb{L}} % Lazard ring

\newcommand{\fplus}{+_F} % formal group +_F
\newcommand{\fminus}{-_F} % formal group -_F

\newcommand{\lbr}{[\hspace{-1.5pt}[}  % left bracket  
\newcommand{\rbr}{]\hspace{-1.5pt}]}  % right bracket  
\newcommand{\FGR}[3]{#1 \lbr #2 \rbr_{#3}} % formal group ring of #1=ring, #2=Z-module and #3=fgl
\newcommand{\RcF}{\FGR{R}{\cl}{F}} % the same as above, but with R, character lattice, F
\newcommand{\RMF}{\FGR{R}{M}{F}} % the same as above, but with M instead of \Lambda
\newcommand{\IF}{\mathcal{I}_F}    % augmentation ideal 

\newcommand{\sS}{S}
\newcommand{\SW}{\sS_W} % S_W
\newcommand{\SWP}[1]{\sS_{W/W_{#1}}} % S_{W/W_P}
\newcommand{\SWd}{\SW^\star} % S_W^* (fixed points)
\newcommand{\SWPd}[1]{\SWP{#1}^\star} % S_{W/W_P}^* (fixed points)
\newcommand{\qQ}{Q}
\newcommand{\QW}{\qQ_W}
\newcommand{\QWP}[1]{\qQ_{W/W_{#1}}} % S_{W/W_P}
\newcommand{\QWd}{\QW^*} % Q_W^* 
\newcommand{\QWPd}[1]{\QWP{#1}^*} % Q_{W/W_P}^* (fixed points)

\newcommand{\sproj}[2]{p_{#2/#1}}  % projection from S_{W/W_\Xi'} to S_{W/W_{\Xi}}
\newcommand{\sprojB}[1]{p_{#1}}  % projection from S_{W} to S_{W/W_{\Xi}}
\newcommand{\sdiag}[2]{d_{#2/#1}}  % diagonal map from S_{W/W_\Xi} to S_{W/W_{\Xi'}}

\newcommand{\de}{\delta}                   % delta
\newcommand{\aA}{A}          % the A operator
\newcommand{\cA}{\mathcal{A}}          % the A operator but from Q_W to Q_W/W_P etc.
\newcommand{\DcF}{\mathbf{D}}    % formal affine Demazure algebra 
\newcommand{\DcFd}{\DcF^\star}   % the dual of formal affine Demazure alg
\newcommand{\DcFP}[1]{\DcF_{#1}} % formal affine Demazure algebra w.r.t. \Xi
\newcommand{\DcFPd}[1]{\DcF_{#1}^\star} %the dual of formal affine Demazure algebra w.r.t. \Xi
\newcommand{\kp}{\kappa}             % the element (x_\alpha + x_{-\alpha})/x_\alpha x_{-\alpha} 

\newcommand{\cmS}{c_S}        %(equivariant) characteristic map c:S -> D^* 

\newcommand{\unit}{\mathbf{1}}      % the identity of the dual

\newcommand{\act}{\bullet}            % the action of Q_W on Q_W^\star
\newcommand{\rev}{\mathrm{rev}}   % reversed sequence
\DeclareMathOperator{\hh}{\mathtt{h}}             % oriented cohomology
\newcommand{\pt}{\mathrm{pt}}   		%base point
\newcommand{\Sch}{\mathrm{Var}}  %category of schemes for which the cohomology theory is defined.
\newcommand{\TSch}{\operatorname{\mathit{T}-\Sch}} %same as above, but with additional torus action
\newcommand{\GSch}[1]{\operatorname{\mathit{#1}-\Sch}} %same as above, but with additional  action of a group
\newcommand{\one}{1} % unit in cohomology ring

\newcommand{\Schub}[2]{\mathcal{X}_{#2}^{#1}} %Schubert variety
\newcommand{\BottSam}[1]{\hat{\mathcal{X}}_{#1}} %Bott-Samelson desingularization
\newcommand{\BS}[1]{\zeta_{#1}} %class of the Bott-Samelson desingularization
\newcommand{\BSP}[2]{\zeta_{#2}^{#1}} %class of the Bott-Samelson desingularization

\newcommand{\res}[1]{\mathrm{res}_{#1}} % restriction corresponding to change of groups
\newcommand{\resG}[1]{\mathrm{Res}_{#1}} % restriction between G-Schemes 

\newcommand{\cg}{c_g} %geometric equivariant characteristic map c:h_T(pt) -> h_T(G/B) 
\newcommand{\cc}{c} %(non-equivariant) characteristic map c:h_T(pt) -> h(G/B)

\newcommand{\inc}{\imath} % inclusion of fixed points 
\newcommand{\incP}[1]{\inc_{#1}} %inclusion of T-fixed points in G/P

\newcommand{\projal}[1]{\pi_{#1}} % projection map G/B -> G/P_alpha 
\newcommand{\proj}[2]{\pi_{#2/#1}} % projection map G/P' -> G/P 
\newcommand{\projP}[1]{\pi_{#1}} % projection map G/B -> G/P
\newcommand{\fproj}[2]{\rho_{#2/#1}} % projection map W/W_{P'} -> W/W_P 
\newcommand{\pairing}[3]{\langle #1,#2\rangle_{#3}} % projection map G/P' -> G/P 

\newcommand{\mainiso}{\Theta} % iso from S_W^\star h_T((G/B)^T) 
\newcommand{\mainisoP}[1]{\mainiso_{#1}} % iso from h_T((G/P)^T) to ... 

\theoremstyle{plain}
\newtheorem{theo}{Theorem}[section]
\newtheorem{prop}[theo]{Proposition}
\newtheorem{lem}[theo]{Lemma}
\newtheorem{cor}[theo]{Corollary}
\newtheorem{ax}{A}

\theoremstyle{definition}
\newtheorem{defi}[theo]{Definition}
\newtheorem{rem}[theo]{Remark}
\newtheorem{ex}[theo]{Example}
\newtheorem{ass}[theo]{Assumption}

\numberwithin{equation}{section}   % equation numbering is inside the section

\begin{document}
%
%%%%%%%%%%%%%%%%%%%%%%%%%%%%%%%%%%%

\title{Equivariant oriented cohomology \\ of flag varieties}

\author{Baptiste Calm\`es}
\author{Kirill Zainoulline}
\author{Changlong Zhong}

\address{Baptiste Calm\`es, Universit\'e d'Artois, Laboratoire de
  Math\'ematiques de Lens, France}
\email{baptiste.calmes@univ-artois.fr}

\address{Kirill Zainoulline, Department of Mathematics and Statistics,
University of Ottawa, Canada}
\email{kirill@uottawa.ca}

\address{Changlong Zhong, Department of Mathematical and Statistical Sciences,
University of Alberta, Canada}
\email{zhongusc@gmail.com}

\thanks{The first author acknowledges the support of the French Agence Nationale de la Recherche (ANR) under reference ANR-12-BL01-0005, as well as the support of the Alexander von Humboldt Foundation, through the special semester in homotopy theory, in Essen.
The second author was supported by the NSERC Discovery grant  385795-2010, NSERC DAS grant 396100-2010 and the Early Researcher Award (Ontario). 
} 

\subjclass[2010]{14F43, 14M15, 19L41, 55N22, 57T15, 57R85}

\maketitle

%\begin{abstract} 
%\end{abstract}

\setcounter{tocdepth}{3}
\tableofcontents

%%%%%%%%%%%%%%%%%%%%%%%%%%
\section{Introduction}

Given an equivariant oriented cohomology theory $\hh$ over a base field $k$, a split reductive group $G$ over $k$, a maximal torus $T$ in $G$ and a parabolic subgroup $P$ containing $T$, we explain how, as a ring, $\hh_T(G/P)$ can naturally be identified with an algebraic object $\DcFPd{\Xi}$ introduced in \cite{CZZ2}, and which is the dual of a coalgebra defined using exclusively the root datum of $(G,T)$, a set of simple roots $\Xi$ defining $P$ and the formal group law $F$ of $\hh$. In \cite{CZZ2}, we studied the properties of this object and of some related operators by algebraic and combinatorial methods, without any reference to geometry. The present paper is to be considered as a companion paper to \cite{CZZ2}, that justifies the definitions of $\DcFPd{\Xi}$ and of other related algebraic objects or operators by explaining how to match them to equivariant cohomology rings endowed with operators constructed using push-forwards and pull-backs along geometric morphisms. 

This kind of algebraic description was first introduced by Demazure in \cite{Dem73,Dem74} for (non-equivariant) Chow groups and K-theory, and then extended to the respective $T$-equivariant settings by Kostant and Kumar \cite{KK86,KK90}. While the non-equivariant case can easily be recovered out of the equivariant one by base change, the big advantage of the equivariant setting is that the pull-back to $T$-fixed points injects $\hh_T(G/P)$ in a very simple ring: a direct product of a finite number of copies of $\hh_T(\pt)$, where $\pt$ is $\Spec(k)$. This important property was already apparent in Atiyah-Bott \cite{AB84} in the topological context of singular cohomology of complex varieties. With this observation, the goal of so-called ``Schubert calculus'' becomes the identification of the image of this injection, and a good description of classes of geometric interest in this image, \ie Schubert varieties. Some operators are also important and we want to describe them too: if $P'$ is a parabolic subgroup contained in $P$, we want to understand pull-back and push-forward maps between $\hh_T(G/P')$ and $\hh_T(G/Q)$, associated to the natural projection $G/P' \to G/P$.
\medskip

Our main results are:
\begin{enumerate}[leftmargin=*]
\item \label{identification_item} Theorem \ref{identificationP_theo}, identifying $\DcFPd{\Xi}$ with $\hh_T(G/P)$, within the fixed points ring $\SWPd{\Xi}$, a direct product of copies of $\hh_T(\pt)$; 
\item Theorem \ref{image_theo}, giving an intrinsic characterization of the image in the Borel case;
\item Diagram \eqref{pushcube_eq}, describing the push-forward operator mentioned above;  
\item Theorem \ref{Wfixed_theo}, identifying the image of the injective pull-back map $\hh_T(G/P) \to \hh_T(G/B)$ ($B$ is a Borel subgroup) as the subring $\hh_T(G/B)^{W_\Xi}$ of fixed elements under the parabolic Weyl group $W_\Xi$ corresponding to $P$;
\item Lemma \ref{BSclasses_lem}, describing the algebraic elements corresponding to Bott-Samelson classes, \ie fundamental classes of desingularized Schubert varieties.
\item Theorem \ref{pairing_theo}, proving that the pairing defined by product and push-forward to $\hh_T(\pt)$ is non-degenerate;
\item Theorem \ref{BorelPres_theo}, providing a Borel style presentation $\hh_T(\pt)\otimes_{\hh_T(\pt)^W} \hh_T(\pt) \simeq \hh_T(G/B)$ (under some conditions). 
\end{enumerate}

We do not prove these results in that order, though. 
First, we state the properties that we use from equivariant oriented cohomology theories, in section \ref{equivariant_sec}. 
Then, in section \ref{multgroup_sec}, we describe $\hh_T(\pt)$ as the formal group ring $S=\RMF$ of \cite[Def.~2.4]{CPZ}. 
In section \ref{torusP1_sec}, we compute the case of $\hh_T(\PP^1)$ when the action of $T$ on $\PP^1=\AA^2/\Gm$ is induced by a linear action of $T$ on $\AA^2$. 
It enables us to identify the pull-back of Bott-Samelson classes to $T$-fixed points in the Borel case, in section \ref{BS_sec}. By localization, some of these classes generate $\hh_T(G/B)$ and this lets us prove the Borel case of \eqref{identification_item}. The parabolic cases are then obtained in the remaining sections, as well as the Borel style presentation. In the last section, we explain how equivariant groups under subgroups of $T$ (and in particular the trivial group which gives the non-equivariant case) can be recovered out of the equivariant one.  
\medskip

We would like to point out several places where the case of an oriented cohomology theory with an arbitrary formal group law is significantly more complicated than the two classical cases of the additive law (Chow groups) and the multiplicative one ($K$-theory). First of all, in these two classical cases, the formal group law is polynomial and furthermore given by very simple polynomials; it is easy to conceive that the computations increase in complexity with other formal group laws given by powers series with an infinite number of nonzero coefficients. Secondly, in both of these classical cases, the (non-equivariant) cohomology ring of a point is $\ZZ$, which is a regular ring, while in general, this base ring can be arbitrary; in the work of Kostant and Kumar, the fraction field of the $T$-equivariant cohomology ring of the point is used as a crucial tool, but we are forced to invert less elements and use a more subtle localization process, for fear of killing everything in some cases (see the definition of $Q$ from $S$ in section \ref{algcomb_sec}). The positive aspect of this extra difficulty is that it forces us to distinguish what really comes from geometry from artifacts of particular cohomology theories. Thirdly, as Bressler and Evens have shown \cite{BE90}, additive and multiplicative formal group laws are the only formal group laws for which the elements $X_{I_w}$ and $Y_{I_w}$ (see after Def. \ref{Xalpha_defi}) are independent of the choice of a reduced decomposition $I_w$ of $w$. Geometrically, this translates as the fact that for Chow groups or $K$-theory, the class of a Bott-Samelson desingularization corresponding to the reduced decomposition $I_w$ only depends on $w$, and actually is the class of the (possibly singular) Schubert variety corresponding to $w$ in Chow groups and the class of its structural sheaf in $K$-theory. For an arbitrary oriented cohomology theory, for example for algebraic cobordism, this is simply not true: different desingularizations of the same Schubert variety give different classes. This combinatorial/geometric independence is used as a key ingredient in the literature on Chow groups and $K$-theory. For example, see \cite[Thm. 1]{Dem73} and how it is used in \cite[\S 4]{Dem74}; see also \cite[Prop.~4.2]{KK86} and its corollary Prop.~4.3. We conjecture that it is the discovery by Bressler and Evens that it does not hold in general that deterred further development of the Demazure-Kostant-Kumar line of ideas (see the first paragraph on p. 550 in \cite{KK90}) until \cite{CPZ}, in which this non-independence is overcome. The situation is now approximately as follows: in the two classical cases, key objects, such as the cohomology of $G/B$ or the algebra $\DcF$, are naturally equipped with a \emph{canonical} basis indexed by elements of the Weyl group $W$, while in general, there are many possible bases, corresponding to different choices of reduced decompositions for every element of $W$. It increases the complexity of the combinatorics involved, but it is still manageable. 
\medskip

Let us mention some of the literature on cohomology theories that go beyond Chow groups or K-theory. In \cite{HHH}, Harada, Henriques and Holm prove the injectivity of the pull-back to fixed points map and the characterization of its image in the topological context of generalized cohomology theories, under an assumption that certain characteristic classes are prime to each other. Our Theorem \ref{image_theo} gives the precise cases when this happens (and, as all of our statements and proofs, it only relies on algebro-geometric methods, with no input from topology).

In \cite[Thm. 5.1]{KiKr13}, a Borel style presentation of equivariant algebraic cobordism is obtained after inverting the torsion index. The improvement of our Theorem~\ref{BorelPres_theo} is that it applies to any oriented cohomology theory, and that, even over a field of characteristic zero, over which algebraic cobordism is the universal oriented cohomology theory, it gives a finer result than what one would get by specializing from cobordism, as one can see in the case of $K$-theory: the Borel style presentation will always hold in the simply connected case, without inverting the torsion index. 

%%%%%%%%%%%%%%%%%%

\section{Equivariant oriented cohomology theory} \label{equivariant_sec}

In the present section we recall the notion of an equivariant algebraic oriented cohomology theory, essentially by compiling definitions and results of \cite{Des09}, \cite{EG}, \cite{HM}, \cite{KiKr13}, \cite{Kr12}, \cite{LM}, \cite{Pa} and \cite{To}. We present it here in a way convenient for future reference.

\medskip

In this paper, $k$ is always a fixed base field, and $\pt$ denotes $\Spec(k)$. By a variety we mean a reduced separated scheme of finite type over $k$. Let $G$ be a smooth linear algebraic group over $k$, abbreviated as \emph{algebraic group}, and let $\GSch{G}$ be the category of smooth quasi-projective varieties over $k$ endowed with an action of $G$, and with morphisms respecting this action (i.e. $G$-equivariant morphisms). The tangent sheaf $\TB{X}$ of any $X \in \GSch{G}$ is locally free and has a natural $G$-equivariant structure. The same holds for the (co)normal sheaf of any equivariant regular embedding of a closed subscheme. 

\medskip

An \emph{equivariant oriented cohomology theory} over $k$ is an additive contravariant functor $\hh_G$ from the category $\GSch{G}$ to the category of commutative rings with unit for any algebraic group $G$ (for an equivariant morphism $f$, the map $\hh_G(f)$ is denoted by $f^*$ and is called \emph{pull-back}) together with

\begin{itemize}
\item 
a morphism $f_*\colon \hh_G(X) \to \hh_G(Y)$  of $\hh_G(Y)$-modules (called \emph{push-forward}) for any projective morphism $f\colon X \to Y$ in $\GSch{G}$ (here $\hh_G(X)$ is an $\hh_G(Y)$-module through $f^*$)

\item 
a natural transformation of functors $\res{\phi}\colon \hh_{H} \to \hh_{G}\,\circ\, \resG{\phi}$ (called \emph{restriction}) for any
morphism of algebraic groups $\phi\colon G \to H$  (here $\resG{\phi}\colon \GSch{H} \to \GSch{G}$ simply restricts the action of $H$ to an action of $G$ through $\phi$)

\item 
a natural transformation of functors $c^{G}\colon K_G \to \tilde\hh_G$ (called the \emph{total equivariant characteristic class}),
where $K_G(X)$ is the $K$-group of $G$-equivariant locally free
sheaves over $X$ and $\tilde\hh_G(X)$ is the multiplicative group of
the polynomial ring $\hh_G(X)[t]$ (the coefficient at $t^i$ is called
the $i$-th equivariant characteristic class in the theory $\hh$ and is denoted by $c_i^{G}$)
\end{itemize}

that satisfy the following properties 

\begin{ax}[Compatibility for push-forwards]
The push-forwards respect composition and commute with pull-backs for transversal squares (a transversal square is a fiber product diagram with a nullity condition on $\mathrm{Tor}$-sheaves, stated in \cite[Def. 1.1.1]{LM}; in particular, this condition holds for any fiber product with a flat map). 
\end{ax}

\begin{ax}[Compatibility for restriction]
The restriction respects composition of morphisms of groups and commutes with push-forwards.
\end{ax}

\begin{ax}[Localization] \label{loc_ax}
For any smooth closed subvariety $i\colon Z \to X$ in $\GSch{G}$ with open complement $u\colon U \hook X$, the sequence
\[
\hh_G(Z) \tooby{i_*} \hh_G(X) \tooby{u^*} \hh_G(U) \to 0
\]
is exact. 
\end{ax}

\begin{ax}[Homotopy Invariance] \label{hominv_ax}
Let $p\colon X\times \mathbb{A}^n\to X$ be a $G$-equivariant
projection with $G$ acting linearly on $\mathbb{A}^n$. Then the induced pull-back $\hh_G(X)\to
\hh_G(X\times\mathbb{A}^n)$ is an isomorphism.
\end{ax}

\begin{ax}[Normalization] \label{ChernOne_ax} 
For any regular embedding $i\colon D \subset X$ of codimension $1$ in $\GSch{G}$ we have $c_1^{G}(\cO(D))=i_*(\one)$ in $\hh_G(X)$, where $\cO(D)$ is the line bundle dual to the kernel of the map of $G$-equivariant sheaves $\cO \to \cO_D$.
\end{ax}

\begin{ax}[Torsors] \label{torsor_ax}
Let $p\colon X \to Y$ be in $\GSch{G}$ and let $H$ be a closed normal
subgroup of $G$ acting trivially on $Y$ 
such that $p\colon X\to Y$ is a $H$-torsor. Consider the quotient map
$\imath\colon G\to G/H$. Then the composite $p^* \circ
\res{\imath}\colon \hh_{G/H}(Y) \to \hh_{G}(X)$ is an isomorphism.

In particular, if $H=G$ 
we obtain an isomorphism $\hh_{\{1\}}(Y)\simeq \hh_G(X)$ for a $G$-torsor $X$
over $Y$.
\end{ax}

\begin{ax} 
If $G=\{1\}$ is trivial, then $\hh_{\{1\}}=\hh$ defines an algebraic oriented cohomology in the sense of \cite[Def. 1.1.2]{LM}  (except that $\hh$ takes values in rings, not in graded rings) with push-forwards and characteristic classes being as in \cite{LM}.
\end{ax}

\begin{ax}[Self-intersection formula]\label{ChernTan_ax} 
Let $i:Y \subset X$ be a regular embedding of codimension $d$ in $\GSch{G}$. Then the normal bundle to $Y$ in $X$, denoted by $\NB{Y/X}$ is naturally $G$-equivariant and there is an equality $i^*i_*(\one)=c_d^G(\NB{Y/X})$ in $\hh_G(Y)$.
\end{ax}

\begin{ax}[Quillen's formula]\label{FGL_ax}
If $\cL_1$ and $\cL_2$ are locally free sheaves of rank one, then 
\[
c_1(\cL_1 \otimes \cL_2)=c_1(\cL_1)\fplus c_1(\cL_2),
\] 
where $F$ is the formal group law of $\hh$ (here $G=\{1\}$).
\end{ax}

For any $X \in \GSch{G}$ consider the $\gamma$-filtration on
$\hh_G(X)$, whose $i$-th term $\gamma^i \hh_G(X)$ is the ideal of
$\hh_G(X)$ generated by products of equivariant characteristic classes of total degree at least~$i$. 
In particular, a $G$-equivariant locally free sheaf of rank $n$ over $\pt$ is the same thing as an $n$-dimensional $k$-linear representation of $G$, so $\gamma^i\hh_G(\pt)$ is generated by Chern classes of such representations. This can lead to concrete computations when the representations of $G$ are well described. 

We introduce the following important notion

\begin{defi} \label{complete_defi}
An equivariant oriented algebraic cohomology theory is called
\emph{Chern-complete over the point} for $G$,  if the ring
$\hh_G(\pt)$ is separated and complete with respect to the $\gamma$-filtration.
\end{defi}

\begin{rem}
Assume that the ring $\hh_G(\pt)$ is separated for all $G$, and let
$\hh_G(\pt)^\wedge$ be its completion with respect to the
$\gamma$-filtration. We can Chern-complete the equivariant cohomology
theory by tensoring with $-\otimes_{\hh_G(\pt)}\hh_G(\pt)^\wedge$. In this way, we obtain a completed version of the cohomology theory, still satisfying the axioms. Note that this completion has no effect on the non-equivariant groups, since in $\hh(\pt)$, the Chern classes are automatically nilpotent by \cite[Lemma 1.1.3]{LM}. 
\end{rem}

Here are three well-known examples of equivariant oriented cohomology theories.

\begin{ex} \label{Chow_ex}
The equivariant Chow ring functor $\hh_G=\mathrm{CH}_G$ was constructed by Edidin and Graham in \cite{EG}, using an inverse limit process of Totaro \cite{To}. In this case the formal group law is the additive one $F(x,y)=x+y$, the base ring $\mathrm{CH}(\pt)$ is $\ZZ$, and the theory is Chern-complete over the point for any group $G$ by construction.
\end{ex}

\begin{ex} \label{K0_ex}
Equivariant algebraic $K$-theory and, in particular, $K_0$ was constructed by Thomason \cite{Th} (see also \cite{Me} for a good survey). The formal group law is multiplicative $F(x,y)=x+y-xy$, the base ring $K_0(\pt)$ is $\ZZ$, and the theory is \emph{not} Chern complete: for example, $(K_0)_{\Gm}(\pt)\simeq \ZZ[t,t^{-1}]$ with the $\gamma^i$ generated by $(1-t)^i$. Observe that $(K_0)_G(\pt)$ consists of classes of $k$-linear finite dimensional representations of $G$.
\end{ex}

\begin{ex}[Algebraic cobordism] \label{Cobord_ex}
Equivariant algebraic cobordism was defined by Deshpande \cite{Des09}, Malg\'on-L\'opez and Heller \cite{HM} and Krishna \cite{Kr12}. The formal group law is the universal one over $\Omega(\pt)=\LL$ the Lazard ring. The equivariant theory is Chern complete over the point for any group $G$ by construction. 
\end{ex}

By Totaro's process one can construct many examples of equivariant theories, such as equivariant connective K-theory, equivariant Morava K-theories, etc. Moreover, in this way one automatically obtains Chern-complete theories.

\section{Torus-equivariant cohomology of a point} \label{multgroup_sec}

In the present section we show that the completed equivariant oriented cohomology ring
of a point $\hh_T(\pt)$, where $T$ is a split torus, can be identified with the formal group
algebra $\RMF$ of the respective group of characters $M$ (see
Theorem~\ref{ShT_theo}).

\medskip

Let $M$ be a finitely generated free abelian group.
Let $T$ be the Cartier dual of $M$, so $M$ is the group of
characters of $T$.
Let $X$ be a smooth variety over $k$ endowed with a trivial $T$-action. 
Consider the pull-back $p^*\colon\hh_T(\pt) \to \hh_T(X)$ induced by the
structure map.
Let $\gamma_\pt^i\hh_T(X)$ denote the ideal in $\hh_T(X)$ generated
by elements from the image of $\gamma^i\hh_T(\pt)$ under the
pull-back. 
Since any representation of $T$ decomposes as a direct sum of one
dimensional representations, $\gamma^i\hh_T(\pt)$ is generated by 
products of first characteristic classes $c_1^{T}(L_\lambda)$,
$\lambda \in M$. Since characteristic classes commute with pull-backs,
$\gamma_\pt^i\hh_T(X)$ is also generated by products of first
characteristic classes (of pull-backs $p^*L_\lambda$).

\medskip

Let $F$ be a one-dimensional commutative formal group law over a ring $R$. We often write $x+_F y$ (formal addition) for the power series $F(x,y)$ defining $F$. 
Following \cite[\S2]{CPZ} consider 
the formal group algebra $\RMF$. It is an $R$-algebra together with an augmentation map $\RMF \to R$ with kernel denoted by $\IF$, and it is complete with respect to the $\IF$-adic topology. 
Thus \[
\RMF = \varprojlim_i \RMF/\IF^i,\] and it is topologically generated by elements of the form $x_\lambda$, $\lambda \in M$, which satisfy
$x_{\lambda+\mu} = x_\lambda+_F x_\mu$.
By definition  (see \cite[2.8]{CPZ}) the algebra $\RMF$ is universal
among $R$-algebras with an augmentation ideal $I$ and a morphism of
groups $M \to (I,\fplus)$ that are complete with respect to the
$I$-adic topology. The choice of a basis of $M$ defines an
isomorphism \[
\RMF \simeq R\lbr x_1,\ldots,x_n \rbr,\] where $n$ is the rank of $M$.

\medskip

Set $R=\hh(X)$. Then $\hh_T(X)$ is an $R$-algebra together with an
augmentation map $\hh_T(X)\to R$ via the restrictions induced by
$\{1\}\to T\to \{1\}$. The assignment $\lambda \in M \mapsto c_1^{T}(L_\lambda)$
induces a group homomorphism $M\to (I,+_F)$, where $I$ is the
augmentation ideal. Therefore, by the universal property of $\RMF$,
there is a morphism of $R$-algebras \[
\phi\colon \RMF/\IF^i \to \hh_T(X)/\gamma_\pt^i \hh_T(X).\]
We claim that

\begin{lem}
The morphism $\phi$ is an isomorphism.
\end{lem}
\begin{proof}
We proceed by induction on the rank $n$ of $M$. 

For $n=0$, we have $T=\{1\}$, $R=\hh_T(X)$, $\IF^i=\gamma_{\pt}^i\hh_T(X)=\{0\}$ and the map $\phi$ turns
into an identity on $R$.

For rank $n>0$ we choose a basis $\{\lambda_1,\ldots,\lambda_n\}$ of
$M$. Let $\{L_1,\ldots,L_n\}$ be the respective one-dimensional
representations of $T$.  
This gives isomorphisms $M \simeq \ZZ^n$ and $T \simeq \Gm^n$ and
$\Gm^n$ acts on $L_i$ by multiplication by the $i$-th coordinate. 
Let $\Gm^n$ act on $\AA^i$ by multiplication by the last
coordinate. Consider the localization sequence (A\ref{loc_ax})
\[
\hh_{\Gm^n}(X) \longrightarrow \hh_{\Gm^n}(X \times \AA^i) \longrightarrow \hh_{\Gm^n}(X \times (\AA^i \setminus \{0\})) \longrightarrow 0.
\]
After identifying
\[
\hh_{\Gm^n}(X)\isoto \hh_{\Gm^n}(X \times
\AA^i)\text{ and }\hh_{\Gm^{n-1}}(X \times \PP^{i-1}) \isoto
\hh_{\Gm^{n}}(X \times (\AA^i \setminus \{0\}))
\]
 via (A\ref{ChernTan_ax}) and
(A\ref{torsor_ax}), 
we obtain an exact sequence
\[
\hh_{\Gm^n}(X) \tooby{c_1(L_n)^i} \hh_{\Gm^n}(X) \longrightarrow \hh_{\Gm^{n-1}}(X \times \PP^{i-1}) \longrightarrow 0.
\] 
where the first
map is obtained by applying self-intersection
(A\ref{ChernOne_ax}) and homotopy invariance (A\ref{hominv_ax})
properties. 

By definition, all these maps are $R$-linear, and the action of $\Gm^{n-1}$ on $X \times \PP^{i-1}$ is the trivial one.
Since the last map is given by pull-back maps and restrictions
(although not all in the same direction), and since equivariant characteristic classes commute with these, one checks that it sends $c_1(L_i)$ to $c_1(L_i)$ for any $i\leq n-1$ and $c_1(L_n)$ to $c_1(\cO(1))$; this last case holds because $\cO(1)$ on $\PP^{i-1}$ goes (by restriction and pull-back) to the equivariant line bundle on $\AA^i\setminus\{0\}$ with trivial underlying line bundle, but where $\Gm^n$ acts by $\lambda_n$ on fibers.

By the projective bundle theorem, we have $R':=\hh(X \times \PP^{i-1})\simeq R[y]/y^i$ with $c_1(\cO(1))=y$. 
By induction, we obtain for any $i$ an isomorphism \[
\hh_{\Gm^{n-1}}(X \times \mathbb{P}^{i-1})/\gamma_\pt^i \simeq \FGR{R'}{M'}{F}/(\IF')^i,\] where $M'=\ZZ^{n-1}$ and $\IF'$ is the augmentation ideal of $\FGR{R'}{M'}{F}$. Using the isomorphisms $\RMF \simeq R\lbr x_1,\ldots,x_n \rbr$ and $\FGR{R'}{M'}{F}\simeq R'\lbr x_1,\ldots,x_{n-1}\rbr$ induced by the basis of $M$, we are reduced to checking that  
\[
\begin{array}{ccc}
R\lbr x_1,\ldots,x_n\rbr/\IF^i & \longrightarrow & (R[y]/y^i)\lbr x_1,\ldots,x_{n-1}\rbr/\mathcal{J} \\
x_i & \longmapsto & \begin{cases} x_i & \text{if $i \leq n-1$} \\ y & \text{if $i=n$.} \end{cases}
\end{array}
\]
is an isomorphism, when $\mathcal{J}=(\IF')^i+y\cdot
(\IF')^{i-1}+\cdots +y^i$. The latter then follows by definition.
\end{proof}

\begin{rem}
Similar statements can be found in \cite[3.2.1]{HM} or \cite[6.7]{Kr12}, but we gave a full proof for the sake of completeness.
\end{rem}

We obtain a natural map of $R$-algebras 
\[
\hh_T(\pt) \to \varprojlim_i \hh_T(\pt)/\gamma^i \hh_T(\pt) \simeq \varprojlim_i \RMF/\IF^i = \RMF
\]
and, therefore, by the lemma

\begin{theo} \label{ShT_theo}
If $\hh$ is (separated and) Chern-complete over the point for $T$,
then the natural map $\hh_T(\pt) \to \RMF$ is an isomorphism. It sends
the characteristic class $c_1^T(L_\lambda)\in \hh_T(\pt)$ to $x_\lambda\in\RMF$.
\end{theo}

\section{Equivariant cohomology of $\PP^1$} \label{torusP1_sec}

In the present section we compute equivariant cohomology
$\hh_T(\PP(V_1\oplus V_2))$ of a
projective line, where a split torus $T$ acts on one-dimensional representations
$V_1$ and $V_2$ by means of characters $\lambda_1$ and $\lambda_2$. 

\begin{ass}
For the rest of the paper we assume that the equivariant cohomology of the point $\hh_T(\pt)$ is (separated and) complete for the $\gamma$-filtration in the sense of Definition \ref{complete_defi}. 
\end{ass}

Let $X$ be a smooth $T$-variety.
By section \ref{multgroup_sec}, the ring
$\hh_T(X)$ can be considered as a ring over $\sS:=\RMF$ via the
identification $\sS \simeq \hh_T(\pt)$ of Theorem \ref{ShT_theo} and
the pull-back map $\hh_T(\pt) \to \hh_T(X)$. By convention, we'll use
the same notation for an element $u$ of $\sS$ and the element $u\cdot
\one \in \hh_T(X)$, where $\one$ is the unit of $\hh_T(X)$. Thus, for
example, $x_\lambda=c_1^T(L_\lambda)$ in $\hh_T(X)$.

\medskip

Given a morphism $f:X \to Y$ in $\TSch$, the pull-back map $f^*$ is a
morphism of rings over $\sS$ and the push-forward map $f_*$ (when it
exists) is a morphism of $\sS$-modules by the projection formula. 

\begin{rem}
Note
that we are not claiming that $\sS$ injects in $\hh_T(X)$ for all $X
\in \TSch$; it will nevertheless hold when $X$ has a $k$-point that is
fixed by $T$, as most of the schemes considered in this paper have.
\end{rem}

\begin{lem} \label{ipoint_lem}
Let $p: X \to Y$ be a morphism in $\GSch{G}$, with a section $s : Y \to X$. Then for any $u \in \hh_G(Y)$, one has 
\begin{enumerate} 
\item \label{cot_item} $s^*s_*(u\cdot v) = u \cdot s^*s_*(v)$ if $s$ is projective.
\item \label{power_item} $p_*(s_*(u)^n)=u\cdot s^*s_*(u)^{n-1}$ for any $n\geq 1$ if furthermore $p$ is projective.
\end{enumerate}
\end{lem}
\begin{proof}
Part \eqref{cot_item} follows from
\[
s^* s_* (u\cdot v) = s^* s_* \big(s^*p^*(u)\cdot v \big) = s^*\big(p^*(u) \cdot s_*(v)\big) = s^* p^*(u) \cdot s^*s_*(v) = u \cdot s^*s_*(v) 
\]
and part \eqref{power_item} from
\[
p_*(s_*(u)^n)=p_*\Big(s_*(u)\cdot s_*(u)^{n-1}\Big)=p_*\Big(s_*\big(u\cdot s^*(s_*(u)^{n-1})\big)\Big)=u\cdot s^*s_*(u)^{n-1}. \qedhere
\]
\end{proof}
This lemma applies in particular when $p: X \to\pt$ is the structural morphism of $X$ and $s$ is therefore a $G$-fixed point of $X$.
\medskip

We now concentrate on the following setting. Let $\lambda_1$ and $\lambda_2$ be characters of $T$, and let $V_1$ and $V_2$ be the corresponding one dimensional representations of $T$, i.e. $t \in T$ acts on $v \in V_i$ by $t\cdot v= \lambda_i(t) v$. Thus, the projective space $\PP(V_1 \oplus V_2)$ is endowed with a natural $T$-action, induced by the action of $T$ on the direct sum of representations $V_1 \oplus V_2$. Furthermore, the line bundle $\cO(-1)$ has a natural $T$-equivariant structure, that can be described in the following way: The geometric points of the total space of $\cO(-1)$ are pairs $(W,w)$ where $W$ is a rank one sub-vector space of $V_1\oplus V_2$ and $w\in W$. The torus $T$ acts by $t\cdot (W,w) = (t(W),t(w))$.

Two obvious embeddings $V_i\subseteq V_1 \oplus V_2$ induce two
$T$-fixed points closed embeddings $\sigma_1,\sigma_2\colon \pt
\hookrightarrow\PP(V_1 \oplus V_2)$. The open complement to $\sigma_1$
is an affine space isomorphic to $V_1 \otimes V_2^\vee$, with
$T$-action by the character $\lambda_1-\lambda_2$. We set
$\alpha:=\lambda_2-\lambda_1$.
By homotopy
invariance (A\ref{hominv_ax}) applied to the pull-back induced by the
structural morphism of $V_1$, we have $\hh_T(\pt) \isoto \hh_T(V_1)$
with inverse given by the pull-back $\sigma_2^*$ (which actually lands
in $V_1$). The exact localization sequence (A\ref{loc_ax}) can therefore be rewritten as
\[
\xymatrix{
\hh_T(\pt) \ar[r]^-{(\sigma_1)_*} & \hh_T(\PP(V_1 \oplus V_2)) \ar[r]^-{\sigma_2^*} & \hh_T(\pt) \ar[r] & 0 \\
}
\] 
Using the structural map $p: \PP(V_1 \oplus V_2) \to \pt$, we get a splitting $p^*$ of $\sigma_2^*$ and a retract $p_*$ of $(\sigma_1)_*$. Thus, the exact sequence is in fact injective on the left, and we can decompose $\hh_T(\PP(V_1 \oplus V_2))$ using mutually inverse isomorphisms 
\begin{equation} \label{isoP1_eq}
\xymatrix{
\hh_T(\pt) \oplus \hh_T(\pt) \ar@<-1ex>[r]_{((\sigma_1)_*,\
  p^*-(\sigma_1)_*p_*p^*)} & \hh_T(\PP(V_1 \oplus V_2))
\ar@<-1ex>[l]_{\tiny \begin{pmatrix}p_* \\ \sigma_2^* \end{pmatrix}}
}
\end{equation}

\begin{lem} \label{geo_lem}
\begin{enumerate}
\item As $T$-equivariant bundles, we have
$\sigma_i^*(\cO(-1))=V_i$.
\item We have $(\sigma_1)_*(\one) = c_1\big(\cO(1)\otimes p^*(V_2)\big)$ and $(\sigma_2)_*(\one) = c_1(\cO(1)\otimes p^*(V_1))$ in $\hh_T\big(\PP(V_1 \oplus V_2)\big)$.
\item \label{sigmasigma_item} For any $u \in \hh_T(\pt)$, we have $\sigma_1^*(\sigma_1)_*(u)=x_{\alpha} u$, $\sigma_2^*(\sigma_2)_*(u)=x_{-\alpha}u$ and $\sigma_1^*(\sigma_2)_*(u) =\sigma_2^*(\sigma_1)_*(u)=0$.
\end{enumerate}
\end{lem}
\begin{proof}
The first part is easily checked on the geometric points of total spaces and is left to the reader. The second part follows from (A\ref{ChernOne_ax}), given the exact sequence of $T$-equivariant sheaves
\[
0 \to \cO(-1)\otimes p^*(V_2)^\vee \to \cO \to \cO_{\sigma_1} \to 0,
\]
where $\cO_{\sigma_1}$ is the structural sheaf of the closed subscheme given by $\sigma_1$. Again this exact sequence is easy to check and we leave it to the reader.
In the third part, the last equality holds by transverse base change through the empty scheme, while the first two follow from Lemma \ref{ipoint_lem} and 
\[
\sigma_1^*(\sigma_1)_*(\one)=\sigma_1^* c_1\big(\cO(1)\otimes p^*(V_2)\big)=c_1\Big(\sigma_1^*\big(\cO(1)\otimes p^*(V_2)\big)\Big) = c_1\big(V_1^\vee \otimes V_2)=x_{\lambda_2-\lambda_1}.
\]
or a symmetric computation for $\sigma_2^*(\sigma_2)_*(\one)$.
\end{proof}

\begin{lem} \label{pushoneP_lem}
If $x_{\alpha}$ is not a zero divisor in $\sS$, 
then the push-forward \[p_*\colon \hh_T(\PP(V_1 \oplus V_2)) \to
\hh_T(\pt)\text{ satisfies
}p_*(\one)=\tfrac{1}{x_{\alpha}}+
\tfrac{1}{x_{-\alpha}}\] 
(observe that $p_*(1)\in S$ by \cite[3.12]{CPZ}, where it is denoted by $e_{\alpha}$).
\end{lem}
\begin{proof}
By Lemma \ref{geo_lem}, we have
\[
\begin{split}
x_{\alpha} & = c_1(p^*(V_2 \otimes V_1^\vee)) = c_1(\cO(1)\otimes p^*(V_2) \otimes (\cO(1)\otimes p^*(V_1))^\vee) \\
 & = c_1\big(\cO(1)\otimes p^*(V_2) \big) \fminus c_1\big(\cO(1)\otimes p^*(V_1)\big) = (\sigma_1)_*(\one) \fminus (\sigma_2)_*(\one).
\end{split}
\]
By transverse base change, we have $(\sigma_1)_*(\one)\cdot
(\sigma_2)_*(\one) =0$, and therefore \[
(\sigma_1)_*(\one) \fminus (\sigma_2)_*(\one) = (\sigma_1)_*(\one) + \big(\fminus (\sigma_2)_*(\one)\big).\]
Since $x_{\alpha}$ is not a zero divisor in $\sS$, it
suffices to prove that \[
x_{\alpha} \cdot p_*(\one)=\one+\tfrac{x_{\alpha}}{x_{-\alpha}},\] where $\tfrac{x_{\alpha}}{x_{-\alpha}}\in\sS^\times$ is the power series $\tfrac{\fminus (x)}{x}$ applied to $x=x_{-\alpha}$. Now, 
\[
\begin{split}
x_{\alpha}p_*(\one) & = p_*(x_{\alpha}) = p_*\big((\sigma_1)_*(\one)+(\fminus (\sigma_2)_*(\one))\big) \\
& = \one + p_*(\fminus (\sigma_2)_*(\one)) = \one + \tfrac{x_{\alpha}}{x_{-\alpha}}.
\end{split}
\]
where the last equality follows from Lemma \ref{ipoint_lem}, part \eqref{power_item}.
\end{proof}

Let $\sigma = \sigma_1 \sqcup \sigma_2 \colon \pt \sqcup \pt \to \PP(V_1 \oplus V_2)$ be the inclusion of both $T$-fixed points.

\begin{lem} \label{injimageP1_lem}
If $x_{\alpha}$ is not a zero divisor in $\sS$, the
pull-back $\sigma^*$ is injective, and 
\[
\im \sigma^*=\{(u,v)\in \hh_T(\pt) \oplus \hh_T(\pt) \mid
x_{-\alpha}u+x_{\alpha}v \in x_{\alpha}x_{-\alpha}\cdot \hh_T(\pt)\}.
\]
\end{lem}
\begin{proof}
Since $\hh_T(\pt \sqcup \pt) = \hh_T(\pt) \oplus \hh_T(\pt)$ identifies $\sigma^*$ with $(\sigma_1^*,\sigma_2^*)$, it suffices to check that the composition 
\[
\xymatrix{
\hh_T(\pt) \oplus \hh_T(\pt) \ar[r]^*!/u1ex/{\labelstyle ((\sigma_1)_*,\ p^*-(\sigma_1)_*p_*p^*)}_{\simeq} & \hh_T(\PP(V_1 \oplus V_2)) \ar[r]^{\tiny\begin{pmatrix}\sigma_1^* \\ \sigma_2^* \end{pmatrix}} & \hh_T(\pt) \oplus \hh_T(\pt)
}
\]
is injective. Indeed, it is given by the matrix
\[\tiny
\begin{pmatrix}
\sigma_1^*(\sigma_1)_* & \sigma_1^* p^* - \sigma_1^* (\sigma_1)_* p_* p^* \\
\sigma_2^* (\sigma_1)_* & \sigma_2^* p^* - \sigma_2^* (\sigma_1)_* p_* p^*
\end{pmatrix}
=
\begin{pmatrix}
x_{\alpha} & \one - x_{\alpha}\cdot p_*(\one) \\
0 & \one 
\end{pmatrix}
=
\begin{pmatrix}
x_{\alpha} & - \frac{x_{\alpha}}{x_{-\alpha}} \\
0 & \one 
\end{pmatrix}
\]
where in the first equality, we have used $p\circ \sigma_i=\id$, Lemma \ref{geo_lem} part \eqref{sigmasigma_item}, to get the $1$'s and the $0$, and then the projection formula $p_*p^*(u)=u\cdot p_*(\one)$ and Lemma \ref{ipoint_lem} to get $\sigma_1^* (\sigma_1)_* p_* p^*(u)=x_{\alpha}p_*(1)\cdot u$. The last equality holds by Lemma \ref{pushoneP_lem}. 

Finally, the image of this matrix is of the expected form.
\end{proof}

Let $\sS[\tfrac{1}{x_{\alpha}}]$ be the localization of $\sS$ at the
multiplicative subset generated by $x_{\alpha}$. Since
$\tfrac{x_{\alpha}}{x_{-\alpha}}$ is invertible, there is a canonical isomorphism $\sS[\tfrac{1}{x_{\alpha}}] \simeq \sS[\tfrac{1}{x_{-\alpha}}]$. We consider the $S[\tfrac{1}{x_{\alpha}}]$-linear operator 
\[
A\colon \sS[\tfrac{1}{x_{\alpha}}] \oplus \sS[\tfrac{1}{x_{\alpha}}]
\longrightarrow  \sS[\tfrac{1}{x_{\alpha}}]\text{ given by }
(u,v) \mapsto \tfrac{u}{x_{\alpha}} + \tfrac{v}{x_{-\alpha}}.
\]
Note that by the previous lemma, it sends the image of $\sigma^*$ to $\sS$ inside $\sS[\tfrac{1}{x_{\alpha}}]$.

\begin{lem} \label{compP1_lem} If $x_{\alpha}$ is not a zero divisor in $\sS$, the following diagram commutes.
\[
\xymatrix{
\hh_T\big(\PP(V_1\oplus V_2)\big) \ar[d]_{p_*} \ar[r]^{\sigma^*} & \hh_T(\pt) \oplus \hh_T(\pt) & \sS \oplus \sS \ar[l]_-{\simeq} \ar@{}[r]|-{\subseteq} & \sS[\tfrac{1}{x_{\alpha}}] \oplus \sS[\tfrac{1}{x_{\alpha}}] \ar[d]^{A} \\
\hh_T(\pt) &  & \sS \ar[ll]_-{\simeq} \ar@{}[r]|-{\subseteq} & \sS[\tfrac{1}{x_{\alpha}}] \\
}
\]
\end{lem}
\begin{proof}
It suffices to check the equality of the two maps after precomposition by the isomorphism $\hh_T(\pt)\oplus \hh_T(\pt) \to \hh_T\big(\PP(V_1\oplus V_2)\big)$ given in \eqref{isoP1_eq}. Using the matrix already computed in the proof of Lemma \ref{injimageP1_lem}, one obtains that the upper right composition sends $(u,v)$ to $u$. The lower left composition sends $(u,v)$ to
\[
p_*\big((\sigma_1)_*(u)+p^*(v)-(\sigma_1)_*p_*p^*(v)\big) = u+ p_*p^*(v) -p_*p^*(v) =u. \qedhere
\] 
\end{proof}

\section{Algebraic and combinatorial objects} \label{algcomb_sec}

Let us now introduce the main algebraic objects $\DcFd$, $\DcFPd{\Xi}$, $\SWd$ and $\SWPd{\Xi}$ that play the role of algebraic replacements for some equivariant cohomology groups in the remaining of this paper. These objects were discussed in detail in \cite{CZZ} and \cite{CZZ2}, and we only give a brief overview here. Their geometric interpretation will be explained in the next sections. 

\medskip

Let $\RS\hookrightarrow \cl^\vee$, $\alpha\mapsto \alpha^\vee$ be a root datum. The rank of the root datum is the dimension of $\mathbb{Q}\otimes_\ZZ\cl$, and elements in $\RS$ are called roots. The root lattice $\cl_r$ is the subgroup of $\cl$ generated by elements in $\RS$, and the weight lattice  is defined as $$\cl_w=\{\omega\in \mathbb{Q}\otimes_\mathbb{Z}\mid \alpha^\vee(\omega)\in \ZZ \text{ for all }\alpha\in \RS\}.$$ We have $\cl_r\subseteq \cl\subseteq \cl_w$. We always assume that the root datum is {\it semisimple} (the ranks of $\cl$, $\cl_r$, $\cl_w$ are equal and no root is twice any other root). The root datum is called {\it simply connected } (resp. {\it adjoint}) if $\cl=\cl_w$ (resp. $\cl=\cl_r$) and if it is furthermore irreducible of rank $n$, we use the notation $\mathcal{D}_n^{sc}$ (resp. $\mathcal{D}_n^{ad}$) for its Dynkin type, with $\mathcal{D}$ among $A$, $B$, $C$, $D$, $G$, $F$, $E$.

\medskip

The Weyl group $W$ of the root datum is the subgroup of $\mathrm{Aut}_\ZZ(\cl)$ generated by simple reflections 
\[
s_\alpha(\lambda)=\lambda-\alpha^\vee(\lambda)\alpha, ~\lambda\in \cl.
\]
Fixing a set of simple roots $\Pi=\{\alpha_1,...,\alpha_n\}$ induces a partition  $\RS=\RS^+\cup \RS^-$, where $\RS^+$ is the set of positive roots and $\RS^-=-\RS^+$ is the set of negative roots. The Weyl group $W$ is actually generated by $s_i:=s_{\alpha_i}$, $i=1,...,n$.

\medskip

Let $F$ be a one-dimensional commutative formal group law over a commutative ring $R$. Let $\sS=\RcF$. From now on we always assume that 
\begin{ass}
The algebra $\sS$ is $\RS$-regular, that is, $x_\alpha$ is regular in $\sS$ for all $\alpha\in \RS$ (see \cite[Def.~4.4]{CZZ}).
\end{ass}
This holds if $2$ is regular in $R$, or if the root datum does not contain an irreducible component of type $C_k^{sc}$ \cite[Rem.~4.5]{CZZ}. 

\medskip

The action of $W$ on $\cl$ induces an action of $W$ on $\sS$, and let $\SW$ be the $R$-algebra defined as $\sS\otimes_R R[W]$ as an $R$-module, and with product given by 
\[
q\delta_w q'\delta_{w'}=qw(q')\delta_{ww'}, \quad q,q'\in \sS,\; w,w'\in W.
\]
Let $Q=\sS[\frac{1}{x_\alpha}|\alpha\in \RS]$ and $\QW=Q\otimes_\sS\SW$, with ring structure given by the same formula with $q, q' \in Q$. Then $\{\delta_w\}_{w\in W}$ is an $\sS$-basis of $\SW$ and a $Q$-basis of $\QW$. There is an action of $\QW$ on $Q$, restricting to an action of $\SW$ on $\sS$, and given by 
\[
q\delta_w\cdot q'=qw(q'), \quad q,q'\in Q,\; w\in W.
\]

For each $\alpha\in \RS$, we define $\kp_\alpha=\frac{1}{x_\alpha}+\frac{1}{x_{-\alpha}}\in \sS$. 
\begin{defi} \label{Xalpha_defi}
For any $\alpha\in \RS$, let 
\[
X_\alpha=\tfrac{1}{x_\alpha}-\tfrac{1}{x_\alpha}\delta_{s_\alpha}, \quad Y_{\alpha}=\kp_\alpha-X_\alpha=\tfrac{1}{x_{-\alpha}}+\tfrac{1}{x_\alpha}\delta_{s_\alpha},
\]
in $\QW$, respectively called a {\it formal Demazure element} and a {\it formal push-pull element}.
\end{defi}

For each sequence $(i_1,...,i_k)$ with $1\le i_j\le n$, we define $X_I=X_{\alpha_{i_1}}\cdots X_{\alpha_{i_k}}$ and $Y_I=Y_{\alpha_{i_1}}\cdots Y_{\alpha_{i_k}}$. 
\begin{defi}
Let $\DcF$ be the $R$-subalgebra of $\QW$ generated by elements from $\sS$ and the elements $X_\alpha$, $\alpha \in \RS$. 
\end{defi}
Since $\delta_{s_i}=1-x_{\alpha_i}X_{\alpha_i}$, we have $\SW\subseteq \DcF$. By \cite[Prop.~7.7]{CZZ}, $\DcF$ is a free $\sS$-module and for any choice of reduced decompositions $I_w$ for every element $w \in W$ the family $\{X_{I_w}\}_{w\in W}$ is an $\sS$-basis of $\DcF$. 

\medskip

There is a coproduct structure on the $Q$-module $\QW$ defined by 
\[
\QW\otimes_Q\QW \to \QW, ~q\delta_w\mapsto q\delta_w\otimes \delta_w,
\]
with counit $\QW\to Q, q\delta_w\mapsto q$. Here $\QW\otimes_Q\QW$ is the tensor product of left $Q$-modules. By the same formula, one can define a coproduct structure on the $\sS$-module $\SW$. The coproduct on $\QW$ induces a coproduct structure on $\DcF$ as a left $\sS$-module. 

\medskip

On duals $\SWd=\Hom_\sS(\SW,\sS)$, $\DcFd=\Hom_S(\DcF,\sS)$ and $\QWd=\Hom_Q(\QW, Q)$ (notice the different stars $\star$ for $\sS$-duality and $*$ for $Q$-duality), the respective coproducts induce products. In $\SWd$ or $\QWd$, this product is given by the simple formula 
\[
f_v f_w=\Kr_{v,w} f_v
\]
on the dual basis $\{f_v\}_{w \in W}$ to $\{\delta_w\}_{w\in W}$, with $\Kr_{v,w}$ the Kronecker delta.
The multiplicative identity is $\unit=\sum_{v\in W}f_v$. 
Let $\eta$ be the inclusion $\SW \subseteq \DcF$. It induces an $\sS$-algebra map $\eta^\star:\DcFd \to \SWd$, which happens to be injective \cite[Lemma 10.2]{CZZ2}. 
Furthermore, after localization, $\eta_Q: \QW \to Q \otimes_\sS \DcFd$ is an isomorphism and by freeness, we have $Q\otimes_\sS\DcFd \simeq \Hom_Q(Q\otimes_\sS \DcF,Q)$ and thus $Q \otimes_\sS \DcFd \simeq \QWd$, as left $Q$-rings. 

\medskip

There is a $Q$-linear action of the $R$-algebra $\QW$ on $\QWd$ given by 
\[
(z\act f)(z')=f(z'z),\quad z,z'\in \QW, f\in \QWd.
\]
as well as $\sS$-linear actions of $\SW$ on $\SWd$ and of $\DcF$ on $\DcFd$, given by the same formula.
With this action, it is proved in \cite[Theorem 10.13]{CZZ2} that $\DcFd$ is a free $\DcF$-module of rank 1 and any $w \in W$ gives a one-element basis $\{x_\Pi\act f_w\}$ of it, where $x_\Pi=\prod_{\alpha\in \RS^-}x_\alpha$.

The map $\cmS: S \to \DcFd$ sending $s$ to $s \act \unit$ is called the \emph{algebraic (equivariant) characteristic map}, and it is of special importance (see section \ref{charmap_sec}).

\medskip
We now turn to the setting related to parabolic subgroups.
Let $\Xi\subseteq \Pi$ be a subset and let $W_\Xi$ be the subgroup of $W$ generated by the $s_i$ with $\alpha_i\in \Xi$. Let $\RS_\Xi=\{\alpha\in \RS|s_\alpha\in W_\Xi\}$, and define $\RS_\Xi^+=\RS^+\cap \RS_{\Xi}$ and $\RS^-_\Xi=\RS^-\cap \RS_\Xi$. For $\Xi'\subseteq \Xi\subseteq \Pi$, let $\RS^+_{\Xi/\Xi'}=\RS_\Xi^+\backslash \RS_{\Xi'}^+$ and $\RS^-_{\Xi/\Xi'}=\RS_\Xi^-\backslash \RS_{\Xi'}^-$. In $\sS$, we set 
\[
x_{\Xi/\Xi'}=\prod_{\alpha\in \RS^-_{\Xi/\Xi'}}x_\alpha \quad\text{and}\quad x_\Xi=x_{\Xi/\emptyset}.
\]

Let $\SWP{\Xi}$ be the free $\sS$-module with basis $\{\delta_{\bar w}\}_{\bar w\in W/W_\Xi}$ and let $\QWP{\Xi}=Q\otimes_\sS\SWP{\Xi}$ be its localization.

\medskip

As on $\QW$, one can define a coproduct structure on $\QWP{\Xi}$ and $\SWP{\Xi}$, by the same diagonal formula.
Let 
\[\SWPd{\Xi}=\Hom_\sS(\SWP{\Xi},\sS)\quad\text{and}\quad\QWPd{\Xi}=\Hom_Q(\QWP{\Xi},Q)\]
be the respective dual rings of the corings $\SWP{\Xi}$ and $\QWP{\Xi}$.
On the basis $\{f_{\bar v}\}_{\bar v\in W/W_\Xi}$ dual to the basis $\{\de_{\bar w}\}_{\bar w\in W/W_\Xi}$, the unit element is $\unit_\Xi=\sum_{\bar v\in W/W_\Xi}f_{\bar v}$, both in $\SWPd{\Xi}$ and in $\QWPd{\Xi}$.

\medskip

Assume $\Xi' \subseteq \Xi$. Let $\bar{w} \in W/W_{\Xi'}$ and let $\hat{w}$ denote its class in $W/W_{\Xi}$. Consider the projection and the sum over orbits
\[
\begin{array}[t]{rccc}
\sproj{\Xi'}{\Xi} : & \SWP{\Xi'} & \to & \SWP{\Xi} \\
 & \de_{\bar{w}} & \mapsto & \de_{\hat{w}}
\end{array}
\quad
\text{and}
\quad
\begin{array}[t]{rccc}
\sdiag{\Xi'}{\Xi} : & \SWP{\Xi} & \to & \SWP{\Xi'} \\
 & \de_{\hat{w}} & \mapsto & \sum\limits_{\substack{\bar{v} \in W/W_{\Xi'} \\ \hat{v}=\hat{w}}} \de_{\bar{v}}
\end{array} \vspace{-3ex}
\] 
with $\sS$-dual maps
\[
\arraycolsep=.4ex
\begin{array}[t]{rccc}
\sproj{\Xi'}{\Xi}^\star : & \SWPd{\Xi} & \to & \SWPd{\Xi'} \\
 & f_{\hat{w}} & \mapsto & \sum\limits_{\substack{\bar{v} \in W/W_{\Xi'} \\ \hat{v}=\hat{w}}} f_{\bar{v}}
\end{array}
\quad
\text{and}
\quad
\begin{array}[t]{rccc}
\sdiag{\Xi'}{\Xi}^\star : & \SWPd{\Xi'} & \to & \SWPd{\Xi} \\
 & f_{\bar{w}} & \mapsto & f_{\hat{w}}
\end{array}.
\] 
Note that $\sproj{\Xi'}{\Xi}$ respects coproducts, so $\sproj{\Xi'}{\Xi}^\star$ is a ring map while $\sdiag{\Xi'}{\Xi}^\star$ isn't.
\medskip

We set $\sprojB{\Xi}=\sproj{\emptyset}{\Xi}$. 
Let $\DcFP{\Xi}$ denote the image of $\DcF$ via $\sprojB{\Xi}$. The coproduct structure on $\QWP{\Xi}$ induces an $\sS$-linear coproduct structure on $\DcFP{\Xi}$, so its $\sS$-dual $\DcFPd{\Xi}$ has a ring structure.

In summary, we have the following diagram followed by its dualization 
\[
\xymatrix{
\SWP{\Xi'} \ar@{^(->}[r]^-{\eta_{\Xi'}} \ar@{->>}[d]^{\sproj{\Xi'}{\Xi}} & \DcFP{\Xi'} \ar@{^(->}[r] \ar@{->>}[d]^{\sproj{\Xi'}{\Xi}} & \QWP{\Xi'} \ar@{->>}[d]^{\sproj{\Xi'}{\Xi}}\\
\SWP{\Xi} \ar@{^(->}[r]^-{\eta_{\Xi}} & \DcFP{\Xi} \ar@{^(->}[r] & \QWP{\Xi}} 
\qquad 
\xymatrix{
\DcFPd{\Xi'} \ar@{^(->}[r]^-{\eta_{\Xi'}^\star} & \SWPd{\Xi'} \ar@{^(->}[r] & \QWPd{\Xi'} \\
\DcFPd{\Xi} \ar@{^(->}[u]^{\sproj{\Xi'}{\Xi}^\star} \ar@{^(->}[r]^-{\eta_\Xi^\star} & \SWPd{\Xi} \ar@{^(->}[u]^{\sproj{\Xi'}{\Xi}^\star} \ar@{^(->}[r] & \QWPd{\Xi} \ar@{^(->}[u]^{\sproj{\Xi'}{\Xi}^\star}}
\]
in which all horizontal maps become isomorphisms after tensoring by $Q$ on the left.
It will receive a geometric interpretation as Diagram \eqref{pullcube_eq}. 
Moreover, by \cite[Lemma 11.7]{CZZ2}, the image of $\sprojB{\Xi}^\star$ in $\DcFd$ (or $\SWd$, $\QWd$) is the subset of $W_\Xi$-invariant elements. 

\medskip

There is no `$\act$'-action of $\SWP{\Xi}$ on $\SWPd{\Xi}$ because $\SWP{\Xi}$ is not a ring. But since $x_{\Pi/\Xi}\in \sS^{W_\Xi}$, the element $x_{\Pi/\Xi}\act f$ is well-defined for any $f\in \SWPd{\Xi}$ and actually belongs to $\DcFPd{\Xi}$ inside $\SWPd{\Xi}$, by \cite[Lemma 15.3]{CZZ2}. This defines a map $\DcFPd{\Xi} \to \SWP{\Xi}$, interpreted geometrically in Diagram \eqref{fixeddiagramP_eq}.
\medskip

For a given set of representatives of $W_\Xi/W_{\Xi'}$ we define the  
{\em push-pull element} by 
\[
Y_{\Xi/\Xi'}=\sum_{w\in W_{\Xi/\Xi'}}\delta_w\tfrac{1}{x_{\Xi/\Xi'}}\in \QW.
\]
We set 
$Y_{\Xi}=Y_{\Xi/\emptyset}$. If $\Xi=\{\alpha_i\}$, then 
$Y_{\Xi}=Y_{\alpha_i}$. By \cite[Lemma 10.12]{CZZ2}, $Y_\Xi\in \DcF$.

\medskip

Let 
\[
\arraycolsep=.4ex
\begin{array}[t]{rccc}
\aA_{\Xi/\Xi'}: & (\QWd)^{W_{\Xi'}} & \to & (\QWd)^{W_{\Xi}} \\
& f & \mapsto & Y_{\Xi/\Xi'}\act f
\end{array}
\quad
\text{and}
\quad
\begin{array}[t]{rccc}
\cA_{\Xi/\Xi'}: & \QWPd{\Xi'} & \to & \QWPd{\Xi} \\
& f & \mapsto & \sdiag{\Xi'}{\Xi}^\star(\tfrac{1}{x_{\Xi/\Xi'}}\bullet f)
\end{array}
\]
and respectively call them {\it push-pull operator} and {\it push-forward operator}. 
The operator $\cA_{\Xi/\Xi'}$ is actually independent of the choice of representatives \cite[Lem.~6.5]{CZZ2}. We have $A_{\Xi/\Xi'}((\DcFd)^{W_{\Xi'}})= (\DcFd)^{W_\Xi}$ by \cite[Cor.~14.6]{CZZ2} and $\cA_{\Xi/\Xi'}$ induces a map $\cA_{\Xi/\Xi'}:\DcFPd{\Xi'}\to \DcFPd{\Xi}$ by \cite[Lemma 15.1]{CZZ2}. These two operators are related by the commutative diagram on the left below, becoming the one on the right after tensoring by $Q$. 
\[
\xymatrix{
\DcFPd{\Xi'}\ar[r]^-{\sprojB{\Xi'}^\star}_-\simeq\ar[d]_{\cA_{\Xi/\Xi'}} & (\DcFd)^{W_{\Xi'}}\ar@<-2ex>[d]^{A_{\Xi/\Xi'}}\\
\DcFPd{\Xi}\ar[r]^-{\sprojB{\Xi}^\star}_-\simeq & (\DcFd)^{W_\Xi}
}
\hspace{6ex}
\xymatrix{
\QWPd{\Xi'}\ar[r]^-{\sprojB{\Xi'}^\star}_-\simeq\ar[d]_{\cA_{\Xi/\Xi'}} & (\QWd)^{W_{\Xi'}}\ar@<-2ex>[d]^{A_{\Xi/\Xi'}}\\
\QWPd{\Xi}\ar[r]^-{\sprojB{\Xi}^\star}_-\simeq & (\QWd)^{W_\Xi}
}
\]
Again, when $\Xi'=\emptyset$, we set $\aA_\Xi=\aA_{\Xi/\emptyset}$ and $\cA_\Xi=\cA_{\Xi/\emptyset}$.

\section{Fixed points of the torus action}

We now consider a split semi-simple algebraic group $G$ over $k$ containing $T$ as a maximal torus. Let $W$ be the Weyl group associated to $(G,T)$, with roots $\RS \subseteq \cl$. We choose a Borel subgroup $B$ of $G$ containing $T$. It defines a set $\Pi$ of simple roots in $W$. Given a subset $\Xi\subseteq \Pi$, the subgroup generated by $B$ and representatives in $G(k)$ of reflections with respect to roots in $\Xi$ is a parabolic subgroup, denoted by $P_{\Xi}$. The map sending $\Xi$ to $P_\Xi$ is a bijection between subsets of $\Pi$ and parabolic subgroups of $G$ containing $B$. Let $W_{\Xi}$ be the subgroup of $W$ generated by reflections with respect to roots in $\Xi$. We will abuse the notation by also writing $W$ (or $W_\Xi$, etc.) when referring to the constant finite algebraic group over $\pt$ whose set of points over any field is $W$. 

\medskip

For any parabolic subgroup $P$, the quotient variety $G/P$ is projective and we consider it in $\TSch$ by letting $T$ act on $G$ by multiplication on the left. After identifying $W\simeq \mathrm{N}_G(T)/T$,
the Bruhat decomposition says that $G/P=\amalg_{w\in W^\Xi} BwP_\Xi/P_\Xi$, where the union is taken over 
the set $W^\Xi$ of minimal left coset-representatives of $W/W_\Xi$.
The latter induces a bijection between $k$-points of $G/P_\Xi$ that are fixed by $T$ and the set $W^\Xi$ (or $W/W_\Xi$). In particular, fixed $k$-points of $G/B$ are in bijection with elements of $W$.

\medskip

Let $(G/P_\Xi)^T=\amalg_{\bar w \in W/W^\Xi} \pt_{\bar w}$ denote the closed subvariety of $T$-fixed $k$-points, then by additivity there is an $\sS=\hh_T(\pt)$-algebra isomorphism
\[
\Theta_\Xi\colon \hh_T((G/P_\Xi)^T)\stackrel{\simeq}\longrightarrow \prod_{\bar w\in W/W_\Xi}\hh_T(\pt_{\bar w})=\prod_{\bar w\in W/W_\Xi} \sS \cong \SWPd{\Xi}.
\]
If $\Pi=\emptyset$, we denote $\Theta:\hh_T((G/B)^T)=\hh_T(W)\to \SWd.$

\medskip

Let $\incP{\Xi}\colon (G/P_{\Xi})^T \hookrightarrow G/P_\Xi$ denote the (closed) embedding of the $T$-fixed locus, 
and let $\incP{\Xi}^{\bar{w}}\colon \pt_{\bar w} \hookrightarrow G/P_\Xi$ denote the embedding corresponding to $\bar{w}$. 
Given $\Xi'\subseteq \Xi \subseteq \Pi$, we define projections
\[
\proj{\Xi'}{\Xi}\colon G/P_{\Xi'} \to G/P_\Xi \hspace{5ex} \text{and}\hspace{5ex} \fproj{\Xi'}{\Xi}\colon W/W_{\Xi'} \to W/W_{\Xi}
\]
(here we view $W/W_\Xi$ as a variety that is a disjoint union of copies of $\pt$ indexed by cosets).
If $\Xi=\{\alpha\}$ consists of a single simple root $\alpha$, we
omit the brackets in the indices, i.e. we abbreviate $W_{\{\alpha\}}$ as
$W_\alpha$, $P_{\{\alpha\}}$ as $P_{\alpha}$, etc. 
If $\Xi'=\emptyset$, we omit the $\emptyset$ in the notation,
\ie $\proj{\emptyset}{\Xi}=\pi_\Xi$, $\fproj{\emptyset}{\Xi}=\rho_\Xi$, etc. By definition, we have
\begin{equation} \label{finitepush_eq}
\mainisoP{\Xi} \circ (\fproj{\Xi'}{\Xi})_*  = \sdiag{\Xi'}{\Xi}^\star \circ \mainisoP{\Xi'}
\hspace{3ex}\text{and}\hspace{3ex} 
\mainisoP{\Xi'} \circ (\fproj{\Xi'}{\Xi})^*  = \sproj{\Xi'}{\Xi}^\star\circ \mainisoP{\Xi}. 
\end{equation}

\begin{lem} \label{tangent_lem}
Let $w \in W$ be a representative of $\bar{w} \in W/W_\Xi$. 
The pull-pack $(\incP{\Xi}^{\bar{w}})^*\TB{G/P_{\Xi}}$ of the tangent bundle $\TB{G/P_{\Xi}}$ of $G/P_{\Xi}$
is the representation of $T$ (the $T$-equivariant bundle over a point) with weights $\{w(\alpha)\mid \alpha \in \RS^-_{\Pi/\Xi}\}$
(observe that by \cite[Lemma 5.1]{CZZ2} it is independent of the choice of a representative $w$). 
\end{lem}
\begin{proof}
Consider the exact sequence of $T$-representations at the neutral
element $e\in G$
\[
0 \to \TB{P_\Xi,e} \to \TB{G,e} \to \TB{G/P_\Xi,e} \to 0
\]
(it is exact by local triviality of the right $P_{\Xi}$-torsor $G \to
G/P_\Xi$). By definition of the root system associated to $(G,T)$, the
roots $\RS$ are the characters of $\TB{G,e}$. By definition of the
parabolic subgroup $P_{\Xi}$, the characters of $\TB{P_\Xi,e}$ are
$\RS^+\sqcup \RS_\Xi^-$. This proves the lemma when $w=e$. For an arbitrary $w$, we consider the diagram
\[
\xymatrix{
\pt_e \ar[r]^{\incP{\Xi}^e} \ar[dr]_{\incP{\Xi}^{\bar e}} & G \ar[r]^{w\cdot} \ar[d] & G \ar[d] \\
 & G/P_{\Xi} \ar[r]^{w\cdot} & G/P_{\Xi} 
}
\] 
which is $T$-equivariant if $T$ acts by multiplication on the left on the right column and through conjugation by $w^{-1}$ and then by multiplication on the left on the left column. Since $\incP{\Xi}^{\bar{w}}$ is the bottom composite from $\pt_e$ to $G/P_\Xi$, the fiber of $\TB{G/P_\Xi}$ at $\bar{w}$ is isomorphic to its fiber at $e$, but for every character $\alpha$, the action of $T$ is now by $t(v) = \alpha (\bar{w}^{-1}t\bar{w})\cdot v = \alpha(w^{-1}(t))\cdot v = w(\alpha)(t)\cdot v$, in other words by the character $w(\alpha)$.
\end{proof}

\begin{prop} \label{points_prop}
We have $(\incP{\Xi}^{\bar{w}})^*(\incP{\Xi}^{\bar{w}'})_*(\one) = 0$ if $\bar{w} \neq \bar{w}' \in W/W_\Xi$ and 
\[
(\incP{\Xi}^{\bar{w}})^*(\incP{\Xi}^{\bar{w}})_*(\one)= \prod_{\alpha \in \RS^-_{\Pi/\Xi}}\hspace{-2ex} x_{w(\alpha)} = w (x_{\Pi/\Xi}).
\] 
\end{prop}
\begin{proof}
The case $\bar{w} \neq \bar{w}'$ holds by transverse base change through the empty scheme. Since the normal bundle to a point in $G/P_{\Xi}$ is the tangent bundle of $G/P_{\Xi}$ pulled back to that point, and since any $T$-representation splits into one-dimensional ones, the case $\bar{w}=\bar{w}'$ follows from (A\ref{ChernTan_ax}) using Lemma \ref{tangent_lem} to identify the characters. 
\end{proof}

\begin{rem}
Note that in the Borel case, the inclusion of an individual fixed point is local complete intersection as any other morphism between smooth varieties, but not ``global'' complete intersection, in the sense that it is not the zero locus of transverse sections of a globally defined vector bundle. Otherwise, for Chow groups, such a point would be in the image of the characteristic map as a product of first Chern classes, and it isn't for types for which the simply connected torsion index isn't $1$. Locally, on an open excluding other fixed points, it becomes such a product, as the previous proposition shows.  
\end{rem}

\begin{cor} \label{points_cor}
We have $\mainisoP{\Xi} (\incP{\Xi})^*(\incP{\Xi})_*(\one)=x_{\Pi/\Xi} \act \unit_{\Xi}$. 
\end{cor}
\begin{proof}
Since $\incP{\Xi} = \bigsqcup_{\bar w \in W/W_\Xi} \incP{\Xi}^{\bar w}$, we have
\begin{align*}
\mainisoP{\Xi}(\incP{\Xi})^*(\incP{\Xi})_*(\one) & 
= \mainisoP{\Xi}\bigg( \sum_{\bar v,\, \bar w \in W/W_\Xi} \hspace{-2ex} (\incP{\Xi}^{\bar v})^*(\incP{\Xi}^{\bar w})_*(\one)\bigg) 
= \mainisoP{\Xi}\bigg(\sum_{\bar w \in W/W_\Xi} \hspace{-2ex} w \big(x_{\Pi/\Xi}\big) \one_{\pt_{\bar w}}\bigg) \\
& = \sum_{\bar w \in W/W_\Xi} \hspace{-2ex} w \big( x_{\Pi/\Xi} \big) f_{\bar w}  = x_{\Pi/\Xi}\act \unit_\Xi.\qedhere
\end{align*}
\end{proof}

\section{Bott-Samelson classes} \label{BS_sec}

In the present section we describe the Bott-Samelson classes  in the $T$-equivariant cohomology
of $G/P_\Xi$.

\medskip

Let $\Xi \subseteq \Pi$ as before.
For each $\bar{w} \in W/W_\Xi$ consider 
the $B$-orbit $BwP_\Xi/P_\Xi$ of the point in $G/P_\Xi$ corresponding to $\bar{w}$. 
It is isomorphic to the affine
space $\AA^{l(v)}$ where $v\in W^\Xi$ is the representative of $\bar{w}$ of
minimal length $l(v)$.
Its closure
$\overline{BwP_{\Xi}/P_\Xi}$ is
called the Schubert variety at $\bar{w}$ with
respect to $\Xi$ and is denoted by $\Schub{\Xi}{\bar{w}}$.  
If $\Xi=\emptyset$, we write $\Schub{}{w}$ for
$\Schub{\emptyset}{w}$. 
Moreover, by Bruhat decomposition the closed complement of $BwP_\Xi/P_\Xi$ is the union of Schubert varieties $\Schub{\Xi}{\bar{u}}$ with $\bar{u}<\bar{w}$ for the Bruhat order on $W/W_{\Xi}$. 
For any $w \in W$, the projection map $G/B \to G/P_{\Xi}$ induces a
projective map $\Schub{}{w} \to \Schub{\Xi}{w}$. Moreover, if $w \in W^\Xi$, then the projection $\Schub{}{w} \to \Schub{\Xi}{w}$ is (projective and) birational. 

\medskip

The variety $\Schub{\Xi}{\bar{w}}$ is not smooth in general, but it admits nice desingularizations, that we now recall, following \cite{Dem74}. 
Given a sequence of simple reflections $I=(s_1,\ldots,s_l)$ corresponding to simple roots $(\alpha_1,\ldots,\alpha_l)$, the Bott-Samelson desingularization of $\Schub{}{I}$ is defined as
\[
\BottSam{I} = P_{\alpha_1} \times^B P_{\alpha_2} \times^B \cdots \times^B P_{\alpha_l}/B
\]
where $\times^B$ means the quotient by the action of $B$ given on
points by $b\cdot (x,y)=(xb^{-1},by)$. By definition, the
multiplication of all factors induces a map $q_I\colon \BottSam{I} \to
G/B$ which factors through a map $\mu_I\colon \BottSam{I} \to \Schub{}{w(I)}$ where $w(I)=s_1\cdots s_l$. 
It is easy to see that if $I'=(s_1,\ldots,s_{l-1})$, the diagram
\begin{equation} \label{BottDiag_eq}
\xymatrix{
\BottSam{I} \ar[r]^{q_I} \ar[d]_{p'} & G/B \ar[d]^{\projal{\alpha_l}} \\
\BottSam{I'} \ar[r]^-{\projal{\alpha_l}\circ q_{I'}} & G/P_{\alpha_l}
}
\end{equation}
is cartesian, when $p'$ is projection on the first $l-1$ factors. By induction on $l$, the variety $\BottSam{I}$ is smooth projective and the morphism $\mu_I$ is projective.
When furthermore $I$ is a reduced decomposition of $w(I)$, meaning
that it is of minimal length among the sequences $J$ such that
$w(J)=w(I)$, the map $\mu_I$ is birational (still by Bruhat decomposition).
We can compose this map with the projection to get a map $\BottSam{w} \to \Schub{\Xi}{\bar{w}}$ and thus when $w \in W^\Xi$, we obtain a (projective birational) desingularization $\BottSam{w}\to \Schub{\Xi}{\bar{w}}$. 
It shows that, $G/P_{\Xi}$ has a cellular decomposition with
desingularizations, as considered just before \cite[Thm.~8.8]{CPZ}, with cells indexed by elements of $W/W_{\Xi}$.

\begin{rem} \label{TequivSchub_rem}
The flag varieties, the Schubert varieties, their Bott-Samelson
desingularizations and the various morphisms between them that we have
just introduced are all $B$-equivariant when $B$ acts on the left, and
therefore are $T$-equivariant. 
\end{rem}

\begin{defi}
Let $q_I^{\Xi}=\pi_{\Xi} \circ q_I$, let $\BSP{\Xi}{I}$ be the push-forward $(q_I^\Xi)_*(\one)$ in $\hh_T(G/P_{\Xi})$, and let $\BS{I}=\BSP{\emptyset}{I}$ in $\hh_T(G/B)$. 
\end{defi}

Note that by definition, we have $(\pi_{\Xi})_*(\BS{I})= \BSP{\Xi}{I}$.

\begin{lem} \label{gendesing_lem}
For any choice of reduced sequences $\{I_w\}_{w \in W^\Xi}$, the classes $\BSP{\Xi}{I_w}$ generate $\hh_T(G/P_{\Xi})$ as an $\sS$-module.
\end{lem}
\begin{proof}
The proof of \cite[Theorem 8.8]{CPZ} goes through when $\hh$ is replaced by $\hh_T$, since all morphisms involved are $T$-equivariant; it only uses homotopy invariance and localization.
\end{proof}

Let $V_0$ (resp. $V_\alpha$) be the $1$-dimensional representation of $T$ corresponding to the $0$ (resp. $\alpha$) character. Let $\sigma_0$ and $\sigma_\alpha$ be the inclusions of $T$-fixed points corresponding to $V_0$ and $V_\alpha$ in $\PP(V_0 \oplus V_\alpha)$ as in the setting of Section \ref{torusP1_sec}. 

Consider the projection $\pi_\alpha\colon G/B\to G/P_\alpha$.
Given an element $w\in W$, with image $\bar{w}$ in $W/W_{\alpha}$ and
any lifting $w'$ of $w$ in $G$, the fiber over the fixed point
$\incP{\alpha}^{\bar{w}}\colon \pt_{\bar w} \to G/P_{\alpha}$ is $w'P_\alpha/B$.

\begin{lem} \label{P1GP_lem}
There is a $T$-equivariant isomorphism $w' P_\alpha /B \simeq
\PP(V_0\oplus V_{-w(\alpha)})$, such that the the closed fixed point
$\imath^w\colon \pt_w \to w'P_\alpha /B \hookrightarrow G/B$
(resp. $\imath^{ws_\alpha}$) is sent to $\sigma_0\colon \pt \to \PP(V_0 \oplus V_{-w(\alpha)})$ (resp. to $\sigma_{-w(\alpha)}$). 
\end{lem}
\begin{proof}
Multiplication on the left by $w'$ defines an isomorphism $P_\alpha/B \to w'P_\alpha/B$ and it is $T$-equivariant if $T$ acts by multiplication on the left on $w'P_\alpha/B$ and through conjugation by $(w')^{-1}$ and then by multiplication on the left on $P_\alpha/B$. Thus, we can reduce to the case where $w'=e$: the general case follows by replacing the character $\alpha$ by $w(\alpha)$. 

First, let us observe that $\mathrm{PGL}_2$ acts on the projective space $\PP^1$ by projective transformations, i.e.
\[
\overline{\begin{pmatrix}t & b \\ c & d\end{pmatrix}}[x:y] = [tx+by:cx+dy]
\] 
with its Borel subgroup $B_{\mathrm{PGL}_2}$ of upper triangular matrices fixing the point $[1:0]$, which therefore gives an identification $\mathrm{PGL}_2/B_{\mathrm{PGL}_2} \simeq \PP^1$. So, its maximal torus $\Gm$ of matrices such that $b=c=0$ and $d=1$ acts by $t[x:y]=[tx:y]=[x:t^{-1}y]$. Thus, as a $\Gm$-variety, this $\PP^1$ is actually $\PP(V_1\oplus V_{0}) \simeq \PP(V_0 \oplus V_{-1})$.

The adjoint semi-simple quotient of $P_\alpha$ is of rank one, so it is isomorphic to $\mathrm{PGL}_2$. The maximal torus $T$ maps to a maximal torus $\Gm$ and the Borel $B$ to a Borel in this $\mathrm{PGL}_2$. Up to modification of the isomorphism by a conjugation, we can assume that this Borel of $\mathrm{PGL}_2$ is indeed $B_{\mathrm{PGL_2}}$ as above. The map $T \to \Gm$ is $\pm \alpha$ (the sign depends on how the maximal torus of $\mathrm{PGL_2}$ is identified with $\Gm$). Since $P_\alpha/B \simeq \mathrm{PGL}_2/B_{\mathrm{PGL}_2}$, we are done by the $\mathrm{PGL}_2$ case.
\end{proof}

Recall the notation from section \ref{algcomb_sec}.
\begin{lem} \label{pushpulli_lem}
The following diagram commutes.
\[
\xymatrix{
\hh_T(G/B) \ar[d]_{\projal{\alpha}^* (\projal{\alpha})_*} \ar[r]^{\inc^*} & \hh_T(W) \ar[r]^{\mainiso}_{\simeq} & \SWd \ar@{}[r]|-{\subseteq} & \QWd \ar[d]^{\aA_{\alpha}} \\
\hh_T(G/B) \ar[r]^{\inc^*} & \hh_T(W) \ar[r]^{\mainiso}_{\simeq} & \SWd \ar@{}[r]|-{\subseteq} & \QWd \\
}
\]
\end{lem}
\begin{proof}
In view of Lemma \ref{P1GP_lem}, the strategy is to reduce to the case of Lemma \ref{compP1_lem} by restricting to the fiber over one fixed point of $G/P_\alpha$ at a time.

We decompose $\QWd = \bigoplus_{w \in W^\alpha} (\qQ\cdot f_w \oplus \qQ \cdot f_{ws_\alpha})$ and note that $\aA_\alpha$ preserves this decomposition since 
\[
\aA_\alpha(f_w)=\frac{1}{x_{-w(\alpha)}}(f_w + f_{ws_\alpha}), \hspace{5ex}\aA_\alpha(f_{ws_\alpha})=\frac{1}{x_{w(\alpha)}}(f_w + f_{ws_\alpha}) 
\]
and $A_\alpha$ is $\qQ$-linear. It therefore suffices to check the commutativity of the diagram after extending both rows on the right by a projection $\QWd \to\qQ\cdot f_w \oplus \qQ \cdot f_{ws_\alpha}$, for all $w \in W^\alpha$. But then, the composite horizontal maps $\hh_T(G/B) \to \qQ\cdot f_w \oplus \qQ \cdot f_{ws_\alpha}$ factor as
\[
\hh_T(G/B) \to \hh_T(P_\alpha w B/B) \to \hh_T(\pt) \oplus \hh_T(\pt) {\simeq} \sS \oplus \sS \subseteq \sS[\tfrac{1}{x_{w(\alpha)}}] \oplus \sS[\tfrac{1}{x_{w(\alpha)}}] \subseteq \qQ \oplus \qQ.
\]
Using proper base change on the diagram
\[
\xymatrix{
G/B \ar[d]_{\projal{\alpha}} & w' P_\alpha/B \ar[d] \ar@{_{(}->}[l]_-{}\\
G/P_\alpha & \pt \ar@{_{(}->}[l]_-{\imath^{\bar{w}}_\alpha}
}
\]
and identifying $w'P_\alpha/ B$ with $\PP(V_0 \oplus V_{-w(\alpha)})$ by Lemma \ref{P1GP_lem}, we are reduced to proving the commutativity of 
\[
\xymatrix{
\hh_T\big(\PP(V_0\oplus V_{-w(\alpha)})\big) \ar[d]_{p^*p_*} \ar[r]^-{\sigma^*} & \hh_T(\pt) \oplus \hh_T(\pt) & \sS \oplus \sS \ar[l]_-{\simeq} \ar@{}[r]|-{\subseteq} & \sS[\tfrac{1}{x_{w(\alpha)}}] \oplus \sS[\tfrac{1}{x_{w(\alpha)}}] \ar[d]^{A_\alpha} \\
\hh_T\big(\PP(V_0\oplus V_{-w(\alpha)})\big) \ar[r]^-{\sigma^*} & \hh_T(\pt) \oplus \hh_T(\pt) & \sS \oplus \sS \ar[l]_-{\simeq} \ar@{}[r]|-{\subseteq} & \sS[\tfrac{1}{x_{w(\alpha)}}] \oplus \sS[\tfrac{1}{x_{w(\alpha)}} ]}
\]
which immediately reduces to the diagram of Lemma \ref{pushpulli_lem} followed by an obvious commutative diagram involving pull-backs 
\[
\xymatrix{
\hh_T\big(\PP(V_0\oplus V_{-w(\alpha)})\big)  \ar[r]^-{\sigma^*} & \hh_T(\pt) \oplus \hh_T(\pt) & \sS \oplus \sS \ar[l]_-{\simeq} \ar@{}[r]|-{\subseteq} & \sS[\tfrac{1}{x_{w(\alpha)}}] \oplus \sS[\tfrac{1}{x_{w(\alpha)}}] \\
\hh_T(\pt) \ar[u]^{p^*} \ar[ur]_{\Delta} &  & \sS \ar[ll]_-{\simeq} \ar@{}[r]|-{\subseteq} \ar[u]_{\Delta} & \sS[\tfrac{1}{x_{w(\alpha)}}] \ar[u]_{\Delta} \\
}
\]
in which $\Delta$ is the diagonal morphism.
\end{proof}

\begin{lem} \label{BSclasses_lem}
For any sequence $I=(i_1,\ldots,i_l)$, the Bott-Samelson class $\BS{I} \in \hh_T(G/B)$ maps to 
\[
\mainiso \circ \inc^*(\BS{I}) = \aA_{I^{\rev}}\big(x_\Pi\cdot f_e\big)
\] 
in $\SWd$. 
\end{lem}
\begin{proof}
By induction using diagram \eqref{BottDiag_eq}, we have \[
\BS{I}=\projal{\alpha_{i_l}}^*(\projal{\alpha_{i_l}})_*\circ \cdots
\circ \projal{\alpha_{i_1}}^*(\projal{\alpha_{i_1}})_*\circ
(\inc^e)_*(\one).
\] 
Since $\mainiso \inc^*(\inc^e)_*(\one)= x_\Pi\cdot f_e$ by Proposition~\ref{points_prop}, the conclusion follows from Lemma~\ref{pushpulli_lem}.
\end{proof}

\section{Equivariant cohomology of a flag variety} \label{equivcoh_sec}

In the present section we describe the $T$-equivariant cohomology of an arbitrary split flag variety $G/P_\Xi$.

\medskip

First, consider the complete flag variety $G/B$.
\begin{prop} \label{SchubertBasis_prop}
For any choice of reduced decompositions $(I_w)_{w \in W}$, the family $(\BS{I_w})_{w \in W}$ form a basis of $\hh_T(G/B)$ over $\sS=\hh_T(\pt)$. 
\end{prop}
\begin{proof}
By Lemma \ref{BSclasses_lem}, the element $\BS{I_w}$ pulls-back to $\aA_{I^{\rev}}\big(x_\Pi\cdot f_e\big)$ in $\SWd$ and these are linearly independent over $\sS$ by \cite[Theorem 12.4]{CZZ2}. They generate $\hh_T(G/B)$ by Lemma \ref{gendesing_lem}. 
\end{proof}

\begin{theo} \label{identificationB_theo}
The pull-back map to fixed points $\inc^*:\hh_T(G/B) \to \hh_T(W)$ is injective, and the isomorphism $\mainiso: \hh_T(W) \simeq \SWd$, identifies its image to $\DcFd \subseteq \SWd$. 
\end{theo}
\begin{proof}
This follows from Lemma \ref{SchubertBasis_prop} and the fact that the $\aA_{I^{\rev}}\big(x_\Pi\cdot f_e\big)$ form a basis of $\DcFd$ as a submodule of $\SWd$, still by \cite[Theorem 12.4]{CZZ2}.
\end{proof}

\begin{rem}
We do not know a direct geometric proof that $\hh_T(G/B)$ injects into $\hh_T((G/B)^T)$, which is of course well known for Chow groups or $K$-theory. Tracking back this fact in the proof of the previous theorem, one notices that it makes full use of the algebraic description of these cohomology groups, in other words, that the map $\DcFd \to \SWd$ is indeed an injection by \cite[Lemma 10.2]{CZZ2}.
\end{rem}

\begin{cor} \label{lociso_cor}
The pull-back map $\inc^*: \hh_T(G/B) \to \hh_T(W)$ becomes an isomorphism after localization at the multiplicative subset generated by all $x_\alpha$ where $\alpha$ is a root. 
\end{cor}
\begin{proof}
After localization at this subset, the inclusion $\DcFd \subseteq \SWd$ becomes an isomorphism (see \cite[Lemma 10.2]{CZZ2}).
\end{proof}
\begin{lem}The following diagram commutes
\[
\xymatrix{
\hh_T(W)\ar[d]_{\simeq}^{\mainiso} \ar[r]^{\inc_*} & \hh_T(G/B)\ar@{^(->}[r]^{\inc^*}\ar[d]_{\simeq}^{\mainiso} & \hh_T(W)\ar[d]_{\simeq}^{\mainiso} \\
\SWd\ar[r]^{x_\Pi\act (-)} & \DcFd\ar@{^(->}[r]^{\eta^*} &\SWd
}
\]
\end{lem}
\begin{proof}This follows from Corollary \ref{points_cor} and Theorem \ref{identificationB_theo}.
\end{proof}

We now consider an arbitrary flag variety $G/P_\Xi$.

\begin{lem} \label{pushBorel_lem}
The following diagram commutes.
\[
\xymatrix{
\hh_T(G/B) \ar[d]_{(\projP{\Xi})_*} \ar[r]^{\inc^*} & \hh_T(W) \ar[r]^{\mainiso}_{\simeq} & \SWd  \ar@{}[r]|-{\subseteq} & \QWd \ar[d]^{\cA_{\Xi}} \\
\hh_T(G/P_\Xi) \ar[r]^{\incP{\Xi}^*} & \hh_T(W/W_\Xi) \ar[r]^-{\mainisoP{\Xi}}_-{\simeq} & \SWPd{\Xi} \ar@{}[r]|-{\subseteq} & \QWPd{\Xi} \\
}
\]
\end{lem}
\begin{proof}
After tensoring the whole diagram with $\qQ$ over $\sS$, the morphism $\inc^*$ becomes an isomorphism by Corollary \ref{lociso_cor}. The family $\big((\inc^w)_*(\one)\big)_{w \in W}$ is a $\qQ$-basis of $\qQ \otimes_\sS \hh_T(G/B)$, since by Proposition~\ref{points_prop}, $\mainiso\circ \inc^*\circ (\inc^w)_*(\one)$ is $f_w$ multiplied by an element that is invertible (in $\qQ$). It therefore suffices to check the equality of both compositions in the diagram when applied to all $(\inc^w)_*(\one)$ with $w \in W$: 
\[
\begin{split}
\cA_{\Xi}\circ\mainiso\circ \inc^*\circ (\inc^w)_*(\one) & =  \cA_{\Xi}(w(x_\Pi)f_w) = w(x_\Pi) \cA_{\Xi}(f_w) \\
 & \overset{(*)}= w(x_{\Pi/\Xi} ) f_{\bar{w}} = \mainisoP{\Xi} (\incP{\Xi})^* (\incP{\Xi}^{\bar{w}})_*(\one) = \mainisoP{\Xi} (\incP{\Xi})^* (\projP{\Xi})_* (\inc^{w})_*(\one) \\
\end{split}
\]
where equality $(*)$ follows from the definition of $\cA_{\Xi}$.
\end{proof}

\begin{cor} \label{pushpullBorel_cor}
The following diagram commutes.
\[
\xymatrix{
\hh_T(G/B) \ar[d]_{(\projP{\Xi})^*(\projP{\Xi})_*} \ar[r]^{\inc^*} & \hh_T(W) \ar[r]^{\mainiso}_{\simeq} & \SWd \ar@{}[r]|-{\subseteq} & \QWd \ar[d]^{\aA_{\Xi}} \\
\hh_T(G/B) \ar[r]^{\inc^*} & \hh_T(W) \ar[r]^{\mainiso}_{\simeq} & \SWd \ar@{}[r]|-{\subseteq} & \QWd \\
}
\]
\end{cor}
\begin{proof}
Using equation \eqref{finitepush_eq}, one easily checks the commutativity of diagram involving pull-backs 
\[
\xymatrix{
\hh_T(G/B)  \ar[r]^{\inc^*} & \hh_T(W) \ar[r]^{\mainiso}_{\simeq} & \SWd \ar@{}[r]|-{\subseteq} & \QWd \\
\hh_T(G/P_\Xi) \ar[u]^{\projP{\Xi}^*} \ar[r]^{\inc^*} & \hh_T(W/W_\Xi) \ar[r]^-{\mainisoP{\Xi}}_-{\simeq} & \SWPd{\Xi} \ar@{}[r]|-{\subseteq} & \QWPd{\Xi} \ar[u]_{\sprojB{\Xi}^\star} \\
}
\]
where $\sprojB{\Xi}^\star$ is the sum over orbits: $\sprojB{\Xi}^\star(f_{\bar{w}})=\sum_{\bar{v}=\bar{w}} f_v$. The result follows from the combination of this diagram and the one of Lemma \ref{pushBorel_lem}. 
\end{proof}

\begin{lem} \label{BSPclasses_lem}
For any sequence $I=(i_1,\ldots,i_l)$, the Bott-Samelson class $\BSP{\Xi}{I} \in \hh_T(G/P_\Xi)$ maps to 
\[
\mainiso \circ (\incP{\Xi})^*(\BSP{\Xi}{I}) = \cA_\Xi\circ\aA_{I^{\rev}}\big(x_\Pi f_e\big)
\] 
in $\SWd$.
\end{lem}
\begin{proof}
We have
\[
\mainiso (\incP{\Xi})^* (\BSP{\Xi}{I_w}) 
= \mainiso (\incP{\Xi})^* (\projP{\Xi})_*(\BS{I_w})
= \cA_{\Xi}\circ \mainiso \circ\inc^* (\BS{I_w}) = \cA_{\Xi}\circ \aA_{I_w^{\rev}}\big(x_\Pi f_e\big)
\]
using Lemma \ref{pushBorel_lem} and Lemma \ref{BSclasses_lem} for the last two equalities.
\end{proof}

\begin{prop} \label{basisGP_prop}
For any choice of reduced decompositions $(I_w)_{w \in W^{\Xi}}$ for elements minimal in their $W_\Xi$-cosets, the classes $\BSP{\Xi}{I_w}$ form an $\sS$-basis of $\hh_T(G/P_{\Xi})$. 
\end{prop}
\begin{proof}
By Lemma~\ref{gendesing_lem}, the classes $\BSP{\Xi}{I_w}$ generate $\hh_T(G/P_{\Xi})$ as an $\sS$-module. We have 
\[
\mainiso \inc^* (\projP{\Xi})^* (\BSP{\Xi}{I_w}) = \mainiso \inc^* (\projP{\Xi})^*(\projP{\Xi})_*(\BS{I_w}) = \aA_{\Xi} \mainiso \inc^* (\BS{I_w}) = \aA_{\Xi} \aA_{I_w^{\rev}}\big(x_\Pi f_e\big)
\]
and these elements are linearly independent by \cite[Theorem 14.3]{CZZ2}.
\end{proof}

Let $\Xi' \subseteq \Xi \subseteq \Pi$.
\begin{cor} \label{pushsurj_cor}
The push-forward map $(\proj{\Xi'}{\Xi})_*:\hh_T(G/P_{\Xi'}) \to \hh_T(G/P_{\Xi})$ is surjective and the pull-back map $(\proj{\Xi'}{\Xi})^*: \hh_T(G/P_{\Xi}) \to \hh_T(G/P_{\Xi'})$ is injective.
\end{cor}
\begin{proof}
Surjectivity is obvious from the fact that $\BS{I_w}$ maps to the basis element $\BSP{\Xi}{I_{\bar{w}}}$ for any $w \in W^\Xi$ and injectivity can be seen in the proof of Proposition \ref{basisGP_prop}: the elements $\BSP{\Xi}{I_{\bar{w}}}$ stay independent when pulled back all the way to $\hh_T(W)$ through $\hh_T(G/B)$.
\end{proof}

\begin{theo} \label{identificationP_theo}
The pull-back map $\incP{\Xi}^*: \hh_T(G/P_{\Xi}) \to \hh_T(W/W_\Xi)$ is injective and the isomorphism $\mainisoP{\Xi}: \hh_T(W/W_\Xi) \isoto \SWPd{\Xi}$ identifies its image to $\DcFPd{\Xi} \subseteq \SWPd{\Xi}$. 
\end{theo}
\begin{proof}
As seen in the proof of Corollary \ref{pushsurj_cor}, pulling back further to $\hh_T(W)$ is injective, so injectivity of $\incP{\Xi}^*$ is clear. By Lemma \ref{BSPclasses_lem}, for any $w \in W^\Xi$, the Bott-Samelson class $\BSP{\Xi}{I_{w}}$ is sent to $\cA_{\Xi}\aA_{I_w^{\rev}}\big(x_{\Pi/\Xi} f_e\big)$. These elements form a basis of $\DcFPd{\Xi}$ by \cite[Theorem 14.3 and Lemma 15.1]{CZZ2}.
\end{proof}

\begin{cor}
The pull-back map $\incP{\Xi}^*: \hh_T(G/P_{\Xi}) \to \hh_T(W/W_\Xi)$ becomes an isomorphism after localization at the multiplicative subset generated by all $x_\alpha$ where $\alpha$ is a root.
\end{cor}
\begin{proof}
After localization at this subset, the inclusion $\DcFPd{\Xi} \subseteq \SWPd{\Xi}$ becomes an isomorphism (see \cite[Lemma 11.5]{CZZ2}).
\end{proof}

As for $G/B$, we have the following commutative diagram
\begin{equation} \label{fixeddiagramP_eq}
\begin{gathered}
\xymatrix{
\hh_T(W/W_\Xi)\ar[r]^{(\incP{\Xi})_*} \ar[d]_{\simeq}^{\mainisoP{\Xi}} & \hh_T(G/P_{\Xi})\ar@{^(->}[r]^{(\incP{\Xi})^*}\ar[d]_{\simeq}^{\mainisoP{\Xi}} & \hh_T(W/W_\Xi)\ar[d]_{\simeq}^{\mainisoP{\Xi}}\\
\SWPd{\Xi}\ar[r]^{x_{\Pi/\Xi}\act (-)} & \DcFPd{\Xi}\ar@{^(->}[r]^{\eta_\Xi^*} & \SWPd{\Xi}
}
\end{gathered}
\end{equation}

\begin{lem} \label{pushP_lem}
The following diagram commutes.
\[
\xymatrix{
\hh_T(G/P_{\Xi'}) \ar[d]_{ (\proj{\Xi'}{\Xi})_*} \ar[r]^{(\incP{\Xi'})^*} & \hh_T(W/W_{\Xi'}) \ar[r]^{\mainisoP{\Xi'}}_{\simeq} & \SWPd{\Xi'} \ar@{}[r]|-{\subseteq} & \QWPd{\Xi'} \ar[d]^{\cA_{\Xi/\Xi'}} \\
\hh_T(G/P_{\Xi}) \ar[r]^{(\incP{\Xi})^*} & \hh_T(W/W_{\Xi}) \ar[r]^{\mainisoP{\Xi}}_{\simeq} & \SWPd{\Xi} \ar@{}[r]|-{\subseteq} & \QWPd{\Xi} \\
}
\]
\end{lem}
\begin{proof}
By the surjectivity claim in Corollary \ref{pushsurj_cor}, we can precompose the diagram by $\projP{\Xi'}$. Since $\cA_{\Xi}=\cA_{\Xi/\Xi'}\circ \cA_{\Xi'}$, the result follows from Lemma \ref{pushBorel_lem} applied first to $\Xi'$ and then to $\Xi$.
\end{proof}

Summarizing, we have the following commutative diagrams describing the correspondence between the cohomology rings and their algebraic counterparts:
\begin{equation} \label{pullcube_eq}
\begin{gathered}
\xymatrix{
 & \hh_T(W/W_{\Xi'}) \ar[rr]_-{\simeq}^-{\mainisoP{\Xi'}} & & \SWPd{\Xi'} \\
\hh_T(G/P_{\Xi'}) \ar[rr]_-(.70){\simeq}^-(.70){\mainisoP{\Xi'}} \ar@{^{(}->}[ur]^{\incP{\Xi'}^*} & & \DcFPd{\Xi'} \ar@{^{(}->}[ur]^{} & \\
 & \hh_T(W/W_{\Xi}) \ar'[r][rr]_-{\simeq}^-{\mainisoP{\Xi}} \ar@{^{(}->}'[u]^{(\fproj{\Xi'}{\Xi})^*}[uu] & & \SWPd{\Xi} \ar@{^{(}->}[uu]_{(\sproj{\Xi'}{\Xi})^\star} \\
\hh_T(G/P_{\Xi}) \ar[rr]_-{\simeq}^-{\mainisoP{\Xi}} \ar@{^{(}->}[uu]^{(\proj{\Xi'}{\Xi})^*} \ar@{^{(}->}[ur]^{\incP{\Xi}^*} & & \DcFPd{\Xi} \ar@{^{(}->}[uu]^{} \ar@{^{(}->}[ur]^{} & \\
}
\end{gathered}
\end{equation} 
For push-forwards, instead, the morphism $\cA_{\Xi/\Xi'} : \QWPd{\Xi'} \to \QWPd{\Xi}$ induces a map $\cA_{\Xi/\Xi'}:\DcFPd{\Xi'} \to \DcFPd{\Xi}$ by \cite[Lemma 15.1]{CZZ2}, and we have:
\begin{equation} \label{pushcube_eq}
\begin{gathered}
\xymatrix@C=5ex{
 & \hh_T(W/W_{\Xi'}) \ar[rr]_-{\simeq}^-{\mainisoP{\Xi'}} & & \SWPd{\Xi'} \ar@{^{(}->}[r] & \QWPd{\Xi'} \ar@{->>}[dd]^{\cA_{\Xi/\Xi'}} \\
\hh_T(G/P_{\Xi'}) \ar@{->>}[dd]_{(\proj{\Xi'}{\Xi})_*} \ar[rr]_-{\simeq}^-{\mainisoP{\Xi'}} \ar@{^{(}->}[ur]^{\incP{\Xi'}^*} & & \DcFPd{\Xi'} \ar@{^{(}->}[ur]^{} \ar@{->>}[dd]_(.30){\cA_{\Xi/\Xi'}} & \\
 & \hh_T(W/W_{\Xi}) \ar'[r][rr]_-{\simeq}^-{\mainisoP{\Xi}} & & \SWPd{\Xi} \ar@{^{(}->}[r] & \QWPd{\Xi} \\
\hh_T(G/P_{\Xi}) \ar[rr]_-{\simeq}^-{\mainisoP{\Xi}} \ar@{^{(}->}[ur]^{\incP{\Xi}^*} & & \DcFPd{\Xi} \ar@{^{(}->}[ur]^{} & \\
}
\end{gathered}
\end{equation}
Notice that on this diagram, there is no map from $\hh_T(W/W_{\Xi'})$ to $\hh_T(W/W_\Xi)$, nor from $\SWPd{\Xi'}$ to $\SWPd{\Xi}$ because the operator $\cA_{\Xi/\Xi'}$ is not defined at that level.

By \eqref{fixeddiagramP_eq} and the identity $x_{\Pi/\Xi'}=x_{\Pi/\Xi}x_{\Xi/\Xi'}$, we finally have the following.
\begin{equation}\label{pushcube2_eq}
\begin{gathered}
\xymatrix@C=5ex{
& \hh_T(W/W_{\Xi'}) \ar[dl]_-{(\incP{\Xi'})_*} \ar@{->>}'[d][dd]^-{(\fproj{\Xi'}{\Xi})_*}\ar[rr]_-{\simeq}^-{\mainisoP{\Xi'}} & & \SWPd{\Xi'} \ar[dl]_-{x_{\Pi/\Xi'}\act\_}\ar@{->>}[dd]^{(\sdiag{\Xi'}{\Xi})^\star}\ar[r] & \QWPd{\Xi'}\ar@{->>}[dd]^{(\sdiag{\Xi'}{\Xi})^\star}\\
\hh_T(G/P_{\Xi'})\ar[rr]_(.4){\simeq}^(.4){\mainisoP{\Xi'}}\ar@{->>}[dd]_{(\proj{\Xi'}{\Xi})_*} & & \DcFPd{\Xi'}\ar@{->>}[dd]_(.3){\cA_{\Xi/\Xi'}}&&\\
 & \hh_T(W/W_{\Xi'})\ar'[r][rr]_-{\simeq}^-{\mainisoP{\Xi}}\ar[dl]_-{(\incP{\Xi})_*} && \SWPd{\Xi}\ar[dl]^{x_{\Pi/\Xi'}\act\_} \ar[r]& \QWPd{\Xi}\\ 
\hh_T(G/P_{\Xi})\ar[rr]_-{\simeq}^-{\mainisoP{\Xi}} & & \DcFPd{\Xi} & &}
\end{gathered}
\end{equation}

\section{Invariant subrings and push-forward pairings}

We now describe how the Weyl group $W$, as an abstract group, acts on $\hh_T(G/B)$, and how $W_\Xi$-invariant elements of this action are related to $\hh_T(G/P_{\Xi})$.

\medskip

Since the projection $G/T \to G/B$ is an affine bundle, by homotopy invariance the induced pull-back  $\hh_T(G/B) \isoto \hh_T(G/T)$ is an isomorphism. The Weyl group action is easier to describe geometrically on $\hh_T(G/T)$. Since $W\simeq N_G(T)/T$, multiplication on the right by $w\in W$ defines a right action of $W$ on $G/T$, by $T$-equivariant morphisms. Action by induced pull-backs, therefore, defines a left action of $W$ on $\hh_T(G/T)$. Similarly, 
a right action of $W$ on the $T$-fixed points $(G/T)^T=W$ induces a left action of $W$ on $\hh_T(W)$, and the pull-back $\hh_T(G/T) \to \hh_T(W)$ is $W$-equivariant. Identifying $\hh_T(G/T)\simeq \hh_T(G/B)$, we obtain the Weyl group action on $\hh_T(G/B)$ with $\inc^*\colon\hh_T(G/B) \to \hh_T(W)$ being $W$-equivariant. 

\medskip

One easily checks on $\sS$-basis elements $f_w$ that through $\mainiso$, this $W$-action on $\hh_T(W)$ corresponds to the $W$-action on $\SWd$ by the Hecke action $w(z)=\de_w \act z$, as described in \cite[\S 4]{CZZ2} (by definition, we have $\de_w \act f_v = f_{vw^{-1}}$).
 
\begin{theo} \label{Wfixed_theo}
The image of the injective pull-back map $\hh_T(G/P_{\Xi}) \to \hh_T(G/B)$ is $\hh_T(G/B)^{W_\Xi}$. 
\end{theo}
\begin{proof}
In Diagram~\eqref{pullcube_eq}, the upper square is $W$-equviariant. Since $\inc^*$ is both $W$-equivariant and injective, we are reduced to showing that $\sprojB{\Xi}^\star$ identifies $\SWPd{\Xi}$ to $(\SWd)^{W_\Xi}$, which  follows from \cite[Lemma 11.7]{CZZ2}.
\end{proof}

The following theorem generalizes \cite[Proposition 6.5.(i)]{Br97}. According to the irreducible Dynkin types of the group, regularity assumptions on elements of the base ring $R$ (or weaker assumptions on elements in $R\lbr x \rbr$) are needed. They are carefully summarized in \cite[Lemma 2.7]{CZZ2}, but as a first approximation, regularity  in $R$ of $2$, $3$ and divisors of $|\cl_w/\cl_r|$ cover all types, except the $C_n^{sc}$ case, in which one needs $2$ to be invertible.\footnote{Regarding these assumptions, there is a slight omission in the statement of \cite[Proposition 6.5.(i)]{Br97}. One needs to add that no root is divisible in the lattice for the statement to hold integrally. Otherwise, for example, the product of all roots divided by $2$ gives a counter-example in the $C_2^{sc}$ case.} 
\begin{theo} \label{image_theo}
Under the regularity assumptions of \cite[Lemma 2.7]{CZZ2}, the image of the injective pull-back $\inc^*: \hh_T(G/B) \to \hh_T(W)\cong \SWd$ is the set of element $\sum_{w\in W}q_wf_w$ such that $x_\alpha|(q_w-q_{s_\alpha w})$ for all roots $\alpha$. 
\end{theo}
\begin{proof}
If follows from \cite[Theorem 10.7]{CZZ2}.
\end{proof}

We now describe the pairing given by multiplication and then push-forward to the point, that we call the \emph{push-forward pairing}.
Let 
\[
\begin{array}{ccc}
\hh_T(G/P_{\Xi})\otimes_{\sS} \hh_T(G/P_{\Xi}) & \tooby{\raisebox{1ex}{$\pairing{-}{-}{\Xi}$}} & \sS \\
\xi \otimes \xi' & \longmapsto & \pairing{\xi}{\xi'}{\Xi} = (\proj{\Xi}{\Pi})_*(\xi \cdot \xi')
\end{array}
\]
It is clearly $\sS$-bilinear and symmetric. Through the isomorphism $\mainiso$, this pairing corresponds to
\[
\pairing{\xi}{\xi'}{\Xi} = \cA_{\Pi/\Xi}(\mainisoP{\Xi}(\xi)\cdot \mainisoP{\Xi}(\xi'))
\]
by Diagram \eqref{pushcube_eq}.

\begin{theo} \label{pairing_theo}
The push-forward pairing $\hh_T(G/P_{\Xi})\otimes_\sS \hh_T(G/P_{\Xi}) \to \hh_T(\pt) \simeq \sS$, sending $(\xi,\xi')$ to $\pairing{\xi}{\xi'}{\Xi}$ is non-degenerate.
\end{theo}
\begin{proof}
This follows from \cite[Theorem 15.6]{CZZ2}.
\end{proof}

\begin{rem}
Note that in \cite[Theorem 15.5]{CZZ2}, we describe a basis that is dual to the basis of Bott-Samelson classes for the push-forward pairing on $G/B$. That dual basis can be very useful for algorithmic computations. However, it is given in combinatorial terms, and we do not have a geometric interpretation of its elements. When the formal group law is additive, this problem disappears since the basis is auto-dual (up to a permutation), see \cite[Prop. 1, p. 69]{Dem74}, but for general formal group laws, this is not the case. 
\end{rem}

\section{Borel style presentation} \label{charmap_sec}

The \emph{geometric (equivariant) characteristic map} $c_g:\hh_T(\pt) \to \hh_T(G/B)$ is defined as the composition
\[
\hh_T(\pt) \isoto \hh_{T \times G}(G) \isofrom \hh_G(G/T) \to \hh_T(G/T) \isofrom \hh_T(G/B)
\]
where the first two maps are isomorphisms from Axiom (A\ref{torsor_ax}), the third is the restriction to the subgroup $T$ of $G$ and the fourth is the pull-back map, an isomorphism by Axiom (A\ref{hominv_ax}) of homotopy invariance. In $\hh_{T \times G}(G)$, the action of $T \times G$ on $G$ is by $(t,g)\cdot g' = gg't^{-1}$, and the other non-trivial actions are by multiplication on the left. 
Note that $\cg$ is $R=\hh(\pt)$-linear, although not $\hh_T(\pt)$-linear. By restricting further to $\hh(G/B)$, one obtains the non-equivariant characteristic map $\cc:\hh_T(\pt) \to \hh(G/B)$. Recall the algebraic characteristic map $\cmS:\sS \to \DcFd$, sending $s \to s \act \unit$, defined in section \ref{algcomb_sec}.
\begin{lem}
The algebraic and geometric characteristic maps coincide with each other, up to the identifications $\sS \simeq \hh_T(\pt)$ of Theorem \ref{ShT_theo} and $\mainiso: \hh_T(G/B) \simeq \DcFd$ of Theorem \ref{identificationB_theo}. 
\end{lem}
\begin{proof}
It suffices to show the equality after embedding in $\SWd \simeq \hh_T(W)$, which decomposes as copies of $\sS$. In other words, it suffices to compare, for every $w \in W$, a map $\phi_w$ from $\sS$ to itself, and a map $\psi_w$ from $\hh_T(\pt)$ to itself. Both are continuous $R$-algebra maps, $\psi_w$ for the topology induced by the $\gamma$-filtration and $\phi_w$ for the $\IF$-adic topology, which correspond to each other through $\sS \simeq \hh_T(\pt)$. Since $\sS$ is (topologically) generated by elements $x_\lambda$, corresponding to first Chern classes of line bundles $c_1^T(L_\lambda)$ in $\hh_T(\pt)$, it suffices to compare $\phi_w(x_\lambda)$ and $\psi_w(c_1^T(L_\Lambda))$. By definition of $\cmS$, we have $\phi(x_\lambda)=x_{w(\lambda)}$. Since $\cg$ is defined using only pull-back and restriction maps, both commuting with taking Chern classes, it suffices to verify that when $\hh=K$, the Grothendieck group, we have $\psi_w([L_\lambda])=[L_{w(\lambda)}]$. This is easily checked by using total spaces of bundles, and the formalism of points. For this purpose, let us consider the following equivariant bundles:
\begin{itemize}
\item $M_\lambda$, the $T \times G$-equivariant line bundle over $G$, whose total space is $L_\lambda \times G$ mapping by the second projection to $G$, and with action given on points by $(t,g) \cdot (v,g')=(\lambda(t)v,gg't^{-1})$;
\item $N_\lambda$, the $G$-equivariant line bundle over $G/T$, whose total space is $G \times^T L_\lambda$, the quotient of $G \times L_\lambda$ by the relation $(gt,v)=(g,\lambda(t)v)$, mapping to $G/T$ by the first projection, and with $G$ action by $g \cdot (g',v)=(gg',v)$; 
\item $M'_\lambda$, the $T \times G$-equivariant line bundle over $G$, whose total space is $G \times_{G/T} G \times^T L_\lambda$, mapping to $G$ by the first projection, with action of $T \times G$ given by $(t,g)\cdot (g_1,g_2,v)=(gg_1t^{-1},gg_2,v)$.
\end{itemize}
It is clear that $L_\lambda$ restricts to $T \times G$ and pulls-back over $G$ to $M_\lambda$. Similarly, $N_\lambda$ restricts and pulls-back to $M'_\lambda$. But $M_\lambda$ maps isomorphically to $M'_\lambda$ by the map $(v,g) \mapsto (g,g,v)$. Therefore, $[L_\lambda]$ maps to $[N_\lambda]$ by the map $K_T(\pt) \isoto K_{T \times G}(G) \isofrom K_{G}(G/T)$. Furthermore, $N_\lambda$ restricts and pulls-back as a $T$-equivariant bundle to the fixed point $w$ in $G/T$ (or $G/B$) as $wT \times^T L_\lambda$ with $T$-action on the left, isomorphic to $L_{w(\lambda)}$. This completes the proof. 
\end{proof}

Let $\tor$ be the torsion index of the root datum, as defined in \cite[\S 5]{Dem73}. See also \cite[5.1]{CPZ} for a table giving the values of its prime divisors for each simply connected type. For other types, one just needs to add the prime divisors of $|\Lambda_w/\Lambda|$ by \cite[\S 5, Prop.~6]{Dem73}.
Together with the previous lemma, \cite[Thm.~11.4]{CZZ} immediately implies a Borel style presentation of $\hh_T(G/B)$. Let $\pi:G/B \to \pt$ be the structural map. 
\begin{theo} \label{BorelPres_theo}
If $2 \tor$ is regular in $R$, then the map $\hh_T(pt)\otimes_{\hh_T(\pt)^W} \hh_T(\pt) \to \hh_T(G/B)$ sending $a \otimes b$ to $\pi^*(a) \cg(b)$ is an $\hh_T(\pt)$-linear ring isomorphism if and only if the (non-equivariant) characteristic map $\cc: \hh_T(\pt) \to \hh(G/B)$ is surjective.
\end{theo}
In particular, it will hold for $K$-theory, since the characteristic map is always surjective for $K$-theory. It will also hold for any cohomology theory if $\tor$ is invertible in $R$, as \cite[Cor.~13.9]{CPZ} shows that the non-equivariant characteristic map is then surjective. 

As mentioned in the introduction, this presentation was obtained in \cite{KiKr13} for algebraic cobordism, with the torsion index inverted, and by using comparisons with complex cobordism. 

\section{Subgroups of $T$}

Let $H$ be a subgroup of $T$ given by the embedding $h:H \hook T$. For example $H$ could be the trivial group, a finite multiplicative group or a subtorus of $T$. For any $X \in \TSch$, and thus in $\GSch{H}$ by restriction, there is a restriction ring map $\res{h}:\hh_T(X) \to \hh_H(X)$, in particular if $X=\pt$, which induces a canonical morphism $\hh_H(\pt) \otimes_{\hh_T(\pt)} \hh_T(X) \to \hh_H(X)$ of rings over $\hh_H(\pt)$, sending $a \otimes b$ to $a \cdot \res{h}(b)$. This ``change of coefficients'' morphism is compatible with pull-backs and push-forwards.

\begin{lem}
The morphism $\hh_H(\pt) \otimes_{\hh_T(\pt)} \hh_T(X) \to \hh_H(X)$ is an isomorphism when $X=G/P_{\Xi}$ or $X=W/W_{\Xi}$.
\end{lem} 
\begin{proof}
The case of $X=W/W_{\Xi}$ is obvious, since as as scheme, it is simply a disjoint union of copies of $\pt$. If $X=G/P_{\Xi}$, the left-hand side is free, with a basis of Bott-Samelson classes. So is the right-hand side: it is still generated as an $\hh_H(\pt)$-module by the corresponding Bott-Samelson classes because the proof of Lemma \ref{gendesing_lem} works for $H$ as well as for $T$. Thus, the change of coefficients is surjective. The push-forward pairing is perfect and commutes to the restriction map from $T$ to $H$, so these classes stay independent in $\hh_H(G/P_{\Xi})$ (they have a dual family). Thus, the change of coefficients is injective. 
\end{proof}

This shows that Diagram \eqref{pullcube_eq} for $H$ is obtained by change of coefficients, as well as Diagram \eqref{pushcube_eq} and Diagram \eqref{pushcube2_eq} except their rightmost columns involving $\qQ$. Theorem \ref{pairing_theo} on the bilinear pairing stays valid for $H$ instead of $T$.

\end{document}